\documentclass[11pt, reqno]{amsart}

\usepackage[a4paper, top=1.7cm, bottom=1.7cm, hmargin=3.2cm, twoside=false]{geometry}

\usepackage{amsmath, amssymb, amsthm}

\usepackage{xcolor}
\definecolor{revisionpurple}{RGB}{95,45,160}

\usepackage[colorlinks=true, urlcolor=blue, linkcolor=blue, citecolor=blue, pdfstartview=FitH]{hyperref}

\newtheorem{theorem}{Theorem}[section]
\newtheorem{proposition}[theorem]{Proposition}
\newtheorem{lemma}[theorem]{Lemma}
\newtheorem{corollary}[theorem]{Corollary}

\newtheorem{remark}[theorem]{Remark}
\newtheorem{example}[theorem]{Example}

\newcommand{\vol}{\operatorname{vol}}

\title [Jacobian estimates and geometric inequalities]{Jacobian estimates and geometric inequalities: intermediate Ricci curvature and quadratic decay}
\author[K.-K. Kwong]{Kwok-Kun Kwong}
\address{School of Mathematics and Physics\\
University of Wollongong\\
Northfields Avenue, Wollongong, NSW 2522\\
Australia}
\email{kwongk@uow.edu.au}

\author[J. Lee]{Jihye Lee}
\address{School of Mathematical and Physical Sciences\\
Macquarie University\\
Sydney, NSW 2109\\
Australia}
\email{jihye.lee@mq.edu.au}

\author[F. Ricci]{Fabio Ricci}
\address{Department of Mathematics\\
University of California, Santa Cruz\\
1156 High Street\\
Santa Cruz, CA 95064\\
USA}
\email{fabio@ucsc.edu}

\subjclass[2020]{53C20, 53C21, 53C40}

\keywords{Jacobian estimates, intermediate Ricci curvature,
normal exponential map, tube-volume comparison, Willmore inequality, Sobolev inequality}

\date{}
\begin{document}
\begin{abstract}

We prove a Jacobian estimate underlying both the Heintze--Karcher comparison and the Alexandrov--Bakelman--Pucci method for submanifolds of arbitrary codimension. The estimate retains the contributions of ambient curvature along geodesics emanating from the submanifold and requires no curvature sign assumption. As applications, we obtain a quantitative Fenchel--Willmore inequality with explicit curvature remainders under nonnegative $n$-intermediate Ricci curvature, Michael--Simon Sobolev and isoperimetric inequalities under nonnegative $(n-1)$-intermediate Ricci curvature, and corresponding extensions under quadratic curvature decay. In each result, the required intermediate Ricci curvature condition depends only on the dimension of the submanifold and is independent of its codimension.
\end{abstract}

\maketitle

\section{Introduction}
\label{sec:introduction}
The Jacobian of the normal exponential map is the central quantity in the classical Heintze--Karcher volume comparison theorem for submanifolds. Integrating its Jacobian estimate over the normal bundle yields tube-volume bounds and, in turn, a Willmore inequality; see, for example, \cite{HeintzeKarcher1978}.
A closely related mechanism underlies the Alexandrov--Bakelman--Pucci (ABP) approach to Sobolev inequalities for submanifolds, where Jacobian estimates for gradient--normal exponential maps on suitable contact sets are combined with the area formula to obtain geometric inequalities \cite{Brendle2023}.

The aim of this paper is to place these arguments in a common framework by establishing a unified Jacobian estimate for gradient--normal exponential maps that retains the curvature contributions along each geodesic.
The resulting approach yields quantitative geometric inequalities and adapts to the setting of quadratic curvature decay.

Let $\left(M^{n+m},\overline g\right)$ be a complete Riemannian manifold and $\Sigma^n\subset M$ be a smooth immersed submanifold.
For $u\in C^2(\Sigma)$ and $t>0$, consider the gradient--normal exponential map
$$
\Phi_t^u:T^\perp\Sigma\longrightarrow M,
\qquad
\Phi_t^u(x,y)
:=
\exp_x\bigl(t(\nabla^\Sigma u(x)+y)\bigr).
$$
On the associated contact set, the differential of $\Phi_t^u$ admits a factorization through the differential of the exponential map $\exp_x:T_xM\to M$ at $x$; see \eqref{eq:factorization DPhi^t_u} and Proposition~\ref{prop:gradient-normal-decomposition}. Taking determinants separates the Jacobian determinant of $\Phi_t^u$ into a tangential factor and the Jacobian determinant of the exponential map. The former is controlled by the Laplacian of $u$, the mean curvature vector of $\Sigma$, and an $n$-dimensional curvature trace along the corresponding geodesic (Proposition~\ref{prop:gradient-normal-main}). The latter is controlled by a weighted Ricci curvature term (Lemma~\ref{lem:gradient-normal-signed-ambient}). Combining these two estimates yields one of the main results of this paper, the general Jacobian estimate in Theorem~\ref{thm:gradient-normal-signed-main}. To avoid introducing the technical notation, we refer the reader to that result for the precise statement.

Let us describe the important special case $u\equiv0$ in slightly more detail.
In this case, the map $\Phi_t^u$ reduces to the normal exponential map. The general Jacobian estimate then yields a quantitative Heintze--Karcher comparison for the normal polar Jacobian.
More precisely, let
$$
\Phi_\Sigma(x,\omega,t):=\exp_x(t\omega),
\qquad
x\in\Sigma,\quad \omega\in S_x^\perp\Sigma,\quad t>0,
$$
be the normal exponential map in polar coordinates (or simply normal polar map), and let $J_\Sigma(x,\omega,t)$ denote its Jacobian with respect to
$d\operatorname{vol}_\Sigma\,d\omega\,dt$.
Set $\gamma_{x,\omega}(s):=\exp_x(s\omega)$, and let $E_1(s),\ldots,E_n(s)$ be the parallel translates along $\gamma_{x,\omega}$ of an orthonormal basis of $T_x\Sigma$.
Define
\begin{equation}\label{eq:R_n intro}
\mathcal{R}_n(x,\omega,s) := \sum_{i=1}^n \left\langle \overline{R}\bigl(\gamma_{x,\omega}'(s),E_i(s)\bigr)E_i(s), \gamma_{x,\omega}'(s)   \right\rangle
\end{equation}
and
$$ \mathcal{B}_{\overline{\operatorname{Ric}}}(x,\omega,t) := \int_0^t s\left(1-\frac{s}{t}\right) \overline{\operatorname{Ric}} \bigl(\gamma_{x,\omega}'(s),\gamma_{x,\omega}'(s)\bigr) \,ds. $$
With $\sigma:=\frac{H}{n}$ denoting the normalized mean curvature vector, our Jacobian estimate takes the following form before the normal cut time (Corollary \ref{cor:quantitative-Heintze-Karcher}):
$$
J_\Sigma(x,\omega,t)
\leq
t^{m-1}
e^{-\mathcal{B}_{\overline{\operatorname{Ric}}}(x,\omega,t)}
\left[
1-t\langle\sigma(x),\omega\rangle
-\frac{t}{n}
\int_0^t
\left(1-\frac{s}{t}\right)^2
\mathcal{R}_n(x,\omega,s)
\,ds
\right]_{}^n.
$$
When $\Sigma$ is closed,
integrating this estimate over the unit normal bundle $S^\perp \Sigma$ yields a quantitative tube-volume comparison; see Corollary~\ref{cor:quantitative-tube-volume}.

The two curvature terms in the preceding Jacobian estimate are naturally controlled by the intermediate Ricci curvature.
Recall that, for a unit vector
$v\in T_pM$ and a $k$-dimensional subspace $W\subset v^\perp$, the
$k$-Ricci curvature in the direction $v$ with respect to $W$ is defined by
$$
\overline{\operatorname{Ric}}_k(v, W):= \sum_{i=1}^{k} \left\langle \overline{R}(v, e_i)e_i, v \right\rangle,
$$
where $\{e_i\}_{i=1}^{k}$ is any orthonormal basis of $W$.
We write
$$\overline{\operatorname{Ric}}_k\ge 0$$
if $\overline{\operatorname{Ric}}_k(v, W)\ge 0$
for every $p\in M$, every unit vector $v\in T_pM$, and every $k$-dimensional subspace $W\subset v^\perp$.
It is monotone in the following sense: if
$\overline{\operatorname{Ric}}_k\geq 0$, then
$\overline{\operatorname{Ric}}_\ell\geq 0$ for every $k\leq \ell\leq \dim M-1$.

To relate this notion to the preceding Jacobian estimate, observe that $\mathcal R_n(x,\omega,s)$ defined in \eqref{eq:R_n intro} is the $n$-Ricci curvature in the direction $\gamma_{x, \omega}^{\prime}(s)$ with respect to the parallel transport of $T_x\Sigma$ along $\gamma_{x,\omega}$:
$$
\mathcal{R}_n(x,\omega,s)
=
\overline{\operatorname{Ric}}_n
\bigl(
\gamma_{x,\omega}'(s),
\operatorname{span}\{E_1(s),\ldots,E_n(s)\}
\bigr).
$$
Hence $\overline{\operatorname{Ric}}_n\geq 0$ implies $\mathcal{R}_n\geq 0$.
Moreover, by the monotonicity above,
$\overline{\operatorname{Ric}}_n\geq 0$ implies
$\overline{\operatorname{Ric}}\geq 0$, and therefore
$\mathcal{B}_{\overline{\operatorname{Ric}}}\geq 0$.

The quantitative tube-volume comparison in Corollary~\ref{cor:quantitative-tube-volume} is closely related to Fenchel--Willmore inequalities.
For smooth boundaries of bounded domains in complete noncompact manifolds
with nonnegative Ricci curvature and Euclidean volume growth,
Agostiniani, Fogagnolo, and Mazzieri \cite{AFM}
proved $$ \int_\Sigma |\sigma|^n\,d\operatorname{vol}_\Sigma \geq \theta|\mathbb S^n|, $$ where $\theta$ denotes the asymptotic volume ratio of the ambient manifold. Wang \cite{wang} subsequently gave a short proof based on Heintze--Karcher comparison. Ji and Kwong \cite{JiKwong} extended this result to submanifolds of arbitrary codimension under nonnegative $k$-Ricci curvature, where $k$ depends on both the dimension and the codimension of the submanifold. More recently, Pan and Yi \cite{PanYi} removed this codimension dependence and established the same inequality under
$\overline{\operatorname{Ric}}_n\geq 0$.

Our quantitative tube-volume comparison retains the curvature terms that
are discarded in deriving the standard Fenchel--Willmore inequality. This yields
the following quantitative refinement.

\begin{theorem}[Quantitative Fenchel--Willmore inequality]
\label{thm:intro-quantitative-Willmore}
Let $(M^{n+m},\overline g)$ be a complete noncompact Riemannian manifold
satisfying
$\overline{\operatorname{Ric}}_n\geq 0$,
and let $\Sigma^n\subset M$ be a closed immersed submanifold. Then
$$
\begin{aligned}
\int_\Sigma |\sigma|^n\,d\operatorname{vol}_\Sigma
-\theta|\mathbb S^n|
\geq\;&
\frac{|\mathbb S^n|}{|\mathbb S^{n+m-1}|}
\int_\Sigma\int_{S_x^\perp\Sigma}
\langle-\sigma(x),z\rangle_+^n
\left(
1-e^{-\mathcal B_{\overline{\operatorname{Ric}}}(x,z)}
\right)
\,dS_z\,d\operatorname{vol}_\Sigma
\\
&+
\frac{|\mathbb S^n|}{|\mathbb S^{n+m-1}|}
\int_\Sigma\int_{S_x^\perp\Sigma}
\langle-\sigma(x),z\rangle_+^{n-1}
\mathcal T_n(x,z)
\,dS_z\,d\operatorname{vol}_\Sigma,
\end{aligned}
$$
where $\mathcal B_{\overline{\operatorname{Ric}}}(x,z)$ and
$\mathcal T_n(x,z)$ are defined in
Section~\ref{sec:quantitative-Willmore}; under the curvature assumption
above, the two terms on the right-hand side are nonnegative.
\end{theorem}

The equality case is characterized in
Theorem~\ref{thm:metric-rigidity}. Equality in the
quantitative inequality need not imply equality in the standard
Fenchel--Willmore inequality, as shown by
Example~\ref{ex:equality} on
$\mathbb S^{n+1}\times\mathbb R$.

We also give a complementary derivation that avoids the Jacobian factorization above. Using only a scalar Riccati comparison, we obtain a Jacobian estimate which, under $\overline{\operatorname{Ric}}_n\geq 0$, is sufficient to recover the standard Fenchel--Willmore inequality; see Section~\ref{subsec:direct-riccati}.

We next apply the preceding Jacobian estimate to the gradient--normal
exponential maps arising in the ABP method.
The main point is that, although the corresponding geodesic direction
need not be normal to $\Sigma$, the condition
$\overline{\operatorname{Ric}}_{n-1}\geq 0$ implies
$\mathcal{R}_n^u\geq 0$ and
$\mathcal{B}_{\overline{\operatorname{Ric}}}^u\geq 0$
along the geodesic.
Combining this with the ABP construction yields a Michael--Simon
Sobolev inequality whose curvature assumption depends only on the
dimension of the submanifold, and not on its codimension.

\begin{theorem}[Michael--Simon Sobolev inequality]
Let $(M^{n+m},\overline g)$ be a complete noncompact Riemannian manifold
satisfying
$
\overline{\operatorname{Ric}}_{n-1}\geq 0,
$
and let $\theta$ be its asymptotic volume ratio. Let
$\Sigma^n\subset M^{n+m}$ be a compact submanifold, possibly with smooth
boundary, and let $f$ be a positive smooth function on $\Sigma$. If
$n\geq 2$ and $m\geq 2$, then
$$
\int_\Sigma
\sqrt{|\nabla^\Sigma f|^2+f^2|H|^2}
\,d\operatorname{vol}_\Sigma
+
\int_{\partial\Sigma}
f\,d\operatorname{vol}_{\partial\Sigma}
\geq
n
\left(
\frac{(n+m)|\mathbb B^{n+m}|}
{m|\mathbb B^m|}
\right)^{\frac{1}{n}}
\theta^{\frac{1}{n}}
\left(
\int_\Sigma
f^{\frac{n}{n-1}}
\,d\operatorname{vol}_\Sigma
\right)^{\frac{n-1}{n}}.
$$
\end{theorem}

As applications, we obtain an isoperimetric inequality for minimal submanifolds and an intrinsic diameter estimate under the same intermediate Ricci curvature assumption.

We finally extend the preceding results by allowing a radially decaying
negative part of the intermediate Ricci curvature.
Extensions of geometric
inequalities from nonnegative curvature to asymptotically nonnegative or
radially decaying curvature settings have been studied in several works;
see, for example,
\cite{Abresch1985,ChongLuoLu2025,DongLinLu2024,MaWu2024,LeeRicci2024LogSobolev}
and the references therein. Here we focus on the quadratic-decay condition.
Quadratic curvature decay by itself is a rather weak geometric condition:
Lott and Shen observed that every connected smooth paracompact manifold
admits a complete Riemannian metric with two-sided quadratic sectional-curvature
decay \cite{LottShen2000}.

Let $M^N$ be complete and
noncompact, and let $\lambda:[0,\infty)\to[0,\infty)$ be continuous and
nonincreasing.  Assume
\begin{equation}\label{eq:b1}
b_1:=\int_0^\infty\lambda(t)\,dt<\infty,
\qquad
B:=\limsup_{t\to\infty}t^2\lambda(t)<\infty.
\end{equation}
For $1\leq k\leq N-1$, we say that $M^N$ satisfies asymptotically nonnegative $k$-Ricci curvature with quadratic decay, denoted by $(\mathrm{QD})_k$, with respect to a base point $o\in M$ if
\begin{equation*}
    \overline{\operatorname{Ric}}_k(v,W)
    \geq -k\lambda(d(o,q))
\end{equation*}
for every point $q \in M$, every unit vector $v \in T_qM$, and every $k$-dimensional subspace $W \subset v^\perp$.

Let $h_1$ be the solution of
$$
h_1''(t)=\lambda(t)h_1(t),
\qquad
h_1(0)=0,
\qquad
h_1'(0)=1,
$$
and define the associated generalized asymptotic volume ratio by
$$
\theta_\lambda
:=
\lim_{r\to\infty}
\frac{\operatorname{vol}(B_o(r))}
{N|\mathbb B^N|\int_0^r h_1(t)^{N-1}\,dt}.
$$

Applying the Jacobian factorization to the normal and gradient--normal
exponential maps in the quadratic-decay setting gives the following
Willmore and Michael--Simon inequalities.

\begin{theorem}[Willmore-type inequality under quadratic curvature decay]
Let $M^N$ be a complete noncompact Riemannian manifold satisfying $(\mathrm{QD})_n$, and let $\Sigma^n\subset M^N$ be a closed immersed submanifold. Then
\begin{equation*}
\int_\Sigma|\sigma|^n\,d\operatorname{vol}_\Sigma \geq
\frac{|\mathbb S^n|\theta_\lambda}{L_\Sigma^{N-1}}
-
\mathcal E_{\lambda,\Sigma},
\end{equation*}
where $\mathcal E_{\lambda,\Sigma}\geq 0$ and $L_\Sigma$ are defined in Section~\ref{sec:decay}.
\end{theorem}

The correction term $\mathcal E_{\lambda,\Sigma}$ cannot in general be omitted under $(\mathrm{QD})_n$.
In Remark~\ref{rem:decay-correction-necessary}, the ambient manifold
satisfies this condition and has $\theta_\lambda>0$, but contains a closed
totally geodesic submanifold whose Willmore integral $\int_\Sigma|\sigma|^n\,d\operatorname{vol}_\Sigma$ is zero.

\begin{theorem}[Michael--Simon Sobolev inequality under quadratic decay]
Let $M^N$ be a complete noncompact Riemannian manifold satisfying
$(\mathrm{QD})_{n-1}$, where $N=n+m$, $n\geq2$, and $m\geq2$.
Let $\Sigma^n\subset M^N$ be a compact embedded submanifold,
possibly with smooth boundary. Then every positive smooth function $f$
on $\Sigma$ satisfies
\begin{align*}
&
\int_\Sigma
\sqrt{|\nabla^\Sigma f|^2+f^2|H|^2}\,d\vol_\Sigma
+
\int_{\partial\Sigma}f\,d\vol_{\partial\Sigma}
+
n\mu_\Sigma\int_\Sigma f\,d\vol_\Sigma
\\
&\qquad\ge
n\left(
\frac{N|\mathbb B^N|\theta_\lambda}
{m|\mathbb B^m|L_\Sigma^{N-1}}
\right)^{\frac{1}{n}}
\left(
\int_\Sigma f^{\frac n{n-1}}\,d\vol_\Sigma
\right)^{\frac{n-1}{n}}.
\end{align*}
The quantities $\mu_\Sigma$ and
$L_\Sigma$ are defined in Section~\ref{sec:decay}.
\end{theorem}

Compared with Chong, Luo, and Lu
\cite[Theorem~1.3]{ChongLuoLu2025}, the preceding Sobolev inequality
requires only $(\mathrm{QD})_{n-1}$, rather than a lower quadratic
sectional-curvature bound, and gives improved constants.
Their proof estimates the Jacobian in two stages, first at $t=1$
and then at later times; see \cite[equation~(3.4)]{ChongLuoLu2025}.
Proposition~\ref{prop:quadratic-gradient-normal-jacobian} instead
gives a direct Jacobian bound by estimating the tangential and
ambient factors separately. This avoids the intermediate step
at $t=1$ and the resulting factor $(e^{2b_1}-1)/(2b_1)$ in the
denominator inside the $n$-th root. The coefficient of
$\int_\Sigma f$ also decreases from $2nb_1+1$ to $n\mu_\Sigma$,
where $\mu_\Sigma\le2b_1$.

A second improvement comes from the integration in the proof of
Theorem~\ref{thm:QD-Sobolev}. When integrating over the normal
vectors, we keep the model function evaluated at the distance
traveled by each geodesic, rather than replace it by its value
at the outer radius. After comparison with the reference model,
this gives the same model annulus integral that appears in the
lower volume bound of Lemma~\ref{lem:annulus-decay-limit}.
The integral therefore cancels after normalization. This avoids
the additional growth estimate used to obtain
\cite[equation~(3.7)]{ChongLuoLu2025} and removes the factor
$(1+\sqrt{1+4B})/2$ from the denominator inside the $n$-th root.
The dependence on $L_\Sigma$ and $\theta_\lambda$ remains.

When $\lambda\equiv0$, we have
$\mu_x=\mu_\Sigma=0$, $L_\Sigma=1$, and $\theta_\lambda=\theta$,
so the two quadratic-decay results reduce to their counterparts under nonnegative
intermediate Ricci curvature.

The paper is organized as follows. Section~\ref{sec:HK} develops the Jacobian comparison framework: we first establish a factorization of the differential of the gradient--normal exponential map and derive the resulting Jacobian estimate, then obtain the quantitative Heintze--Karcher and tube-volume comparisons, and finally give a complementary derivation based on a scalar Riccati comparison.
Section~\ref{sec:quantitative-Willmore} applies the tube-volume comparison to prove the quantitative Fenchel--Willmore inequality and characterize its equality case. Section~\ref{sec:isoperimetric} applies the gradient--normal comparison in the ABP framework to obtain the Michael--Simon Sobolev and isoperimetric inequalities under nonnegative intermediate Ricci curvature, together with intrinsic diameter estimates. Section~\ref{sec:decay} extends the Willmore and Michael--Simon inequalities to the quadratic curvature-decay setting.

\noindent
\textbf{Acknowledgements.}
J. Lee is supported by Discovery Projects Grant DP250103808
of the Australian Research Council.
F. Ricci thanks Guofang Wei for her constant support
and helpful discussions during the preparation of this work.
Part of this work was carried out while he was at UC Santa Barbara,
where he was partially supported by NSF grant DMS-2403557.
\section{Jacobian comparison and volume estimates}
\label{sec:HK}
Let $n,m\ge1$, set $N:=n+m$, and let $(M^N,\overline g)$ be a complete Riemannian manifold. Let $\Sigma^n\subset M^N$ be a smooth immersed submanifold, possibly with smooth boundary.  We first factor the differential of the map arising in the Alexandrov--Bakelman--Pucci construction and use this factorization to derive a Jacobian estimate that retains the curvature contributions explicitly. The same estimate applies to the normal exponential map and yields quantitative Heintze--Karcher and tube-volume comparisons. In Subsection~\ref{subsec:direct-riccati}, we give a complementary derivation of the normal Jacobian estimate using a direct Riccati comparison argument, without relying on the factorization.

Throughout this paper, $\overline\nabla$ and $\nabla^\Sigma$ denote the Levi-Civita connections of $M$ and $\Sigma$, respectively, and we use the convention $\overline R(X,Y)Z=\overline\nabla_X\overline\nabla_YZ-\overline\nabla_Y\overline\nabla_XZ-\overline\nabla_{[X,Y]}Z$ for the curvature tensor of $\overline g$. We write $\mathrm{II}(X,Y):=(\overline\nabla_XY)^\perp$ for the second fundamental form of $\Sigma$, where $X,Y$ are tangential vector fields.

\subsection{Factorization of the differential of the gradient--normal exponential map}\label{subsec:main-factor}
Let $u\in C^2(\Sigma)$ and, for $t>0$, define
the gradient--normal exponential map as
\begin{equation*}
\Phi_t^u:T^\perp\Sigma\longrightarrow M,
\qquad
\Phi_t^u(x,y)
:=
\exp_x\bigl(t(\nabla^\Sigma u(x)+y)\bigr).
\end{equation*}
When $u\equiv0$, the map $\Phi_t^u$ reduces to the ordinary fixed-time normal exponential map.

For $r>0$, define the contact set
\begin{align*}
A_r^u
:=
\bigg\{(x,y)\in T^\perp\Sigma|_{\operatorname{int}\Sigma}:\;&
ru(p)
+\frac12 d\bigl(p,\Phi_r^u(x,y)\bigr)^2
\notag\\
&\ge
ru(x)
+\frac{r^2}{2}
\bigl|\nabla^\Sigma u(x)+y\bigr|^2,
\quad\text{for every }p\in\Sigma
\bigg\}
\end{align*}
where $d$ denotes the distance on $M$.

For the geometric motivation and interpretation of this contact set, we refer to
\cite[Definition~1.1 and the discussion thereafter]{wang2013}.
The following lemma shows that the contact condition persists for all earlier times and that the associated radial geodesic is minimizing.

\begin{lemma}
\label{lem:general-contact-propagation}
Let $(x,y)\in A_r^u$, set $v:=\nabla^\Sigma u(x)+y$ and $a:=|v|$, and let
$q_t:=\exp_x(tv)$ for $0\le t\le r$. Then, for every $0<t<r$ and every
$p\in\Sigma$,
\begin{equation}
\label{eq:propagated-contact-condition}
        tu(p)+\frac12d(p,q_t)^2
        \ge
        tu(x)+\frac{t^2}{2}a^2.
\end{equation}
In particular, $x$ is a global minimum point on $\Sigma$ of
$F_t^u(p):=tu(p)+\frac12d(p,q_t)^2$. Moreover, the geodesic
$s\mapsto\exp_x(sv)$ is minimizing from $x$ on $[0,r]$, and therefore
$q_t$ is not in the cut locus of $x$ for $0<t<r$.
\end{lemma}

\begin{proof}
Taking $p=x$ in the definition of $A_r^u$ gives
$d(x,q_r)^2\ge r^2a^2.$
The curve $s\mapsto\exp_x(sv)$ has length $ra$ on $[0,r]$, and hence
$d(x,q_r)\le ra$.  Thus
$d(x,q_r)=ra,$
so this geodesic is minimizing on $[0,r]$.  In particular,
\[
        d(q_t,q_r)=(r-t)a
        \qquad (0<t<r).
\]

Fix $p\in\Sigma$.  By the triangle inequality,
\[
        d(p,q_r)
        \le
        d(p,q_t)+(r-t)a.
\]
The contact condition at time $r$ therefore gives
\begin{align*}
r\bigl(u(p)-u(x)\bigr)
&\ge
\frac12\left(r^2a^2-d(p,q_r)^2\right)
\\
&\ge
\frac12
\left[
        r^2a^2-
        \bigl(d(p,q_t)+(r-t)a\bigr)^2
\right].
\end{align*}
Multiplying by $t/r$ and adding
$\frac12d(p,q_t)^2-\frac12t^2a^2$, we obtain
\begin{align*}
&t\bigl(u(p)-u(x)\bigr)
+\frac12d(p,q_t)^2
-\frac12t^2a^2
\\
\ge&
\frac{t}{2r}
\left[
        r^2a^2-
        \bigl(d(p,q_t)+(r-t)a\bigr)^2
\right]
+\frac12d(p,q_t)^2
-\frac12t^2a^2
\\
=&
\frac{r-t}{2r}
\bigl(d(p,q_t)-ta\bigr)^2
\ge0.
\end{align*}
This is exactly \eqref{eq:propagated-contact-condition}.  The final
assertion follows from the minimizing property proved above.
\end{proof}

Fix $(x,y) \in A_r^u$ and set $v:= \nabla^\Sigma u(x) + y$.
For $0 <t <r$, define
$q_t : = \exp_x (tv)$
and
$$E_t : = D(\exp_x)_{tv} :T_x M \longrightarrow T_{q_t}M.$$
Then by the previous lemma, $E_t$ is nonsingular.
Equip $T^\perp\Sigma$ with the connection metric, so that the normal
connection gives the orthogonal identification
\[
        T_{(x,y)}T^\perp\Sigma
        \simeq
        T_x\Sigma\oplus T_x^\perp\Sigma.
\]
Define
\begin{equation}\label{eq:splitting}
L_t^u: = E_t^{-1}\circ D\Phi_t^u|_{(x,y)}: T_x\Sigma \oplus T_x^\perp \Sigma \longrightarrow T_x M = T_x \Sigma \oplus T_x^\perp \Sigma.
\end{equation}
Then
\begin{equation}\label{eq:factorization DPhi^t_u}
D \Phi_t^u|_{(x,y)} = E_t \circ L_t^u.
\end{equation}
The following proposition gives the basic decomposition of $L_t^u$ used in this paper.

\begin{proposition}
\label{prop:gradient-normal-decomposition}
\begin{enumerate}
    \item With respect to the splitting in \eqref{eq:splitting}, $L_t^u$ has the block form
    \begin{equation}\label{eq:block-gradient-normal-decomposition}
        L_t^u = \begin{pmatrix}
        B_t^u & 0 \\ C_t^u & t I_m
    \end{pmatrix},
    \end{equation}
    where $B_t^u : T_x\Sigma \to T_x\Sigma$ is a symmetric and positive semidefinite linear map.
    Consequently,
    \begin{equation}
    \label{eq:gradient-normal-det-factorization}
        \left|\det D\Phi_t^u(x,y)\right|
        =
        t^m\det B_t^u
        \left|\det D(\exp_x)_{tv}\right|.
    \end{equation}
    \item Moreover, if $a,b \in T_x \Sigma$, then
    \begin{equation}
    \label{eq:gradient-normal-B-index-form}
    \left\langle B_t^ua,b\right\rangle
    =
     tI^t(Y_a,Y_b)
    +t\nabla_\Sigma^2u(a,b)
    -t\langle\mathrm{II}(a,b),y\rangle,
    \end{equation}
    where $Y_a$ is the unique Jacobi field along the geodesic $\gamma(s):=\exp_x(sv)$, $0 \leq s \leq t$, satisfying
    $$Y_a(0) = a, \quad Y_a(t) = 0,$$
    and
    \[
    I^t(U,V)
    :=
    \int_0^t
    \left(
    \langle\overline\nabla_sU,\overline\nabla_sV\rangle
    -
    \left\langle
    \overline R(\gamma',U)V,\gamma'
    \right\rangle
    \right)ds
    \]
    is the index form along the geodesic $\gamma|_{[0,t]}$.
    Equivalently,
    \begin{equation}
    \label{eq:gradient-normal-B-Hessian}
        \left\langle B_t^u a,b\right\rangle
        =
        \nabla_\Sigma^2F_t^u(x)(a,b),
        \qquad a,b\in T_x\Sigma,
    \end{equation}
    where
    \[
        F_t^u(p)=tu(p)+\frac12d(p,q_t)^2.
    \]
    \end{enumerate}

\end{proposition}

\begin{proof}
    We first determine the action of $L_t$ on the vertical directions.
    For a vertical variation $y_\tau=y+\tau\eta$, with
    $\eta\in T_x^\perp\Sigma$, we have
    \[
            \Phi_t^u(x,y_\tau)
            =
            \exp_x(tv+\tau t\eta).
    \]
    Differentiating at $\tau=0$ gives
    \[
        D\Phi_t^u(x,y)(0,\eta)
        = D(\exp_{x})_{tv}(t\eta)=
        E_t(t\eta),
    \]
    and therefore
    \begin{equation}\label{eq:Lt-vertical}
        L_t^u(0,\eta)=t\eta.
    \end{equation}
    For a horizontal vector $a \in T_x\Sigma$, write
    $$L_t^u(a,0) = B_t^ua + C_t^ua,$$
    where $B_t^ua \in T_x\Sigma$ and $C_t^u a \in T_x^\perp \Sigma$.
    Together with \eqref{eq:Lt-vertical}, this gives the block form \eqref{eq:block-gradient-normal-decomposition}.
    It follows that
    $$\det L_t^u = t^m \det B_t^u.$$
    Since $D\Phi_t^u = E_t \circ L_t^u$,
    we obtain \eqref{eq:gradient-normal-det-factorization}.

    It remains to identify $B_t^u$, and in particular to show that it is symmetric and positive semidefinite.
    Because $E_t = D(\exp_x)_{tv}$ is nonsingular, the implicit function theorem applied to
    $$\mathrm{Exp}: TM \longrightarrow M, \quad \mathrm{Exp}(p,w) : = \exp_p (w)$$
    gives a smooth vector field $W$, defined in a neighborhood of $x$ in $M$, such that for all $p$ in this neighborhood,
    $$W(x) = tv, \quad \exp_p(W(p)) = q_t.$$
    Let $Z$ be a local normal vector field along $\Sigma$ satisfying
    $$Z(x) = y, \quad \nabla^\perp Z(x) = 0.$$
    Let $a \in T_x\Sigma$. Choose a smooth curve $x(\tau) \subset \Sigma$ with $x(0) = x$ and $x'(0) = a$.
    Consider the curves in $TM$
    \[
        \mathcal A(\tau)
        :=
        \left(
        x(\tau),
        t\bigl(\nabla^\Sigma u(x(\tau))+Z(x(\tau))\bigr)
        \right), \qquad  \mathcal C(\tau):=(x(\tau),W(x(\tau))).
    \]
    The two curves satisfy
    $ \mathcal A (0) =  \mathcal C(0) = (x,tv)$,
    while
    $\mathrm{Exp}( \mathcal C(\tau)) = q_t.$
    Hence, in the connection splitting of $T_{(x,y)}(TM)\simeq T_x M \oplus T_x M$,
    \[
        \mathcal A'(0)-\mathcal C'(0)
        =
        \left(
        0,
        t\overline\nabla_a(\nabla^\Sigma u+Z)-\overline\nabla_aW
        \right).
    \]
    Since $D \mathrm{Exp} (C'(0)) = 0$, we obtain
    \begin{align*}
        D\Phi_t(x,y) (a,0) & = D \mathrm{Exp}_{(x,tv)}(A'(0)) \\
        &= D \mathrm{Exp}_{(x, tv)} (A'(0) - C'(0))\\
        &= E_t (t\overline\nabla_a(\nabla^\Sigma u+Z)-\overline\nabla_aW.).
    \end{align*}
    Applying $E_t^{-1}$ gives
    \[
            L_t^u(a,0)
            =
            t\overline\nabla_a(\nabla^\Sigma u+Z)-\overline\nabla_aW.
    \]
    Taking the tangential component, we find
    \begin{equation}\label{eq:Bt-preliminary}
        B_t^ua = (t\overline\nabla_a(\nabla^\Sigma u+Z)-\overline\nabla_aW)^T.
    \end{equation}
    We now express the second term in \eqref{eq:Bt-preliminary} through a fixed-endpoint Jacobi field.
    Define
    \[
        \mathcal F(\tau,s)
        :=
        \exp_{x(\tau)}\left(\frac{s}{t}W(x(\tau))\right),
        \qquad 0\le s\le t,
    \]
    and let
    \[
        Y_a(s)
        :=
        \left.\frac{\partial\mathcal F}{\partial\tau}\right|_{\tau=0}.
    \]
    Then $Y_a$ is a Jacobi field along $\gamma(s)$ satisfying
    \[
        Y_a(0)=a,
        \qquad
        Y_a(t)=0,
    \]
    and the torsion-free property gives
    \[
        \overline\nabla_sY_a(0)
        =
        \frac1t\overline\nabla_aW.
    \]
    Therefore
    $$B_t^u a = t(\overline\nabla_a(\nabla^\Sigma u+Z)-\overline\nabla_sY_a)^T.$$
    Pairing this identity with $b \in T_x \Sigma$, and using
    \[
        \langle\overline\nabla_aZ,b\rangle
        =
        -\langle\mathrm{II}(a,b),y\rangle,
    \]
    we obtain
    \begin{equation}\label{eq:Bt-boundary}
        \langle B_t^u a , b\rangle  = t\nabla_\Sigma^2u(a,b)
    -t\langle\mathrm{II}(a,b),y\rangle - t \langle \overline\nabla_sY_a(0) , b\rangle.
    \end{equation}
    Because $Y_a$ and $Y_b$ are Jacobi fields and $Y_b(t) = 0$, integration by parts gives
    $$I^t (Y_a, Y_b) = \langle \overline\nabla_sY_a, Y_b \rangle |_{0}^t = - \langle \overline\nabla_sY_a (0) , b \rangle.$$
    By plugging this into \eqref{eq:Bt-boundary}, we obtain \eqref{eq:gradient-normal-B-index-form}.

    We now verify \eqref{eq:gradient-normal-B-Hessian}.  Since $q_t$ is not in
the cut locus of $x$, the squared-distance function $d_{q_t}^2$ to $q_t$ is smooth near
$x$.  The fixed-endpoint Hessian identity (\cite[Eqn 6.6.9]{Jost2017}) gives
\[
        \frac12\overline\nabla^2d_{q_t}^2(a,b)
        =
        tI^t(Y_a,Y_b).
\]
Moreover,
\[
        \overline\nabla\left(\frac12d_{q_t}^2\right)(x)
        =
        -tv
        =
        -t\bigl(\nabla^\Sigma u(x)+y\bigr).
\]
The restriction formula for Hessians therefore yields
\begin{align*}
\nabla_\Sigma^2F_t^u(a,b)
&=
 t\nabla_\Sigma^2u(a,b)
 +\frac12\overline\nabla^2d_{q_t}^2(a,b)
 +\left\langle
 \overline\nabla\left(\frac12d_{q_t}^2\right),
 \mathrm{II}(a,b)
 \right\rangle
\\
&=
 t\nabla_\Sigma^2u(a,b)
 +tI^t(Y_a,Y_b)
 -t\langle\mathrm{II}(a,b),y\rangle,
\end{align*}
because $\nabla^\Sigma u(x)$ is tangent and $\mathrm{II}(a,b)$ is normal.
This is \eqref{eq:gradient-normal-B-Hessian}.  By
Lemma~\ref{lem:general-contact-propagation}, $x$ is a minimum point of
$F_t^u$, and hence $B_t^u\ge0$.  Symmetry of $B_t^u$ also follows from the above expression because both the index form and the second fundamental form are symmetric.
\end{proof}

In view of \eqref{eq:gradient-normal-det-factorization}, we want to estimate $\det B_t^u$ and $\det D\left(\exp _x\right)_{t v}$. The estimate of $\det B_t^u$ is given by the following proposition.
\begin{proposition}
\label{prop:gradient-normal-main}
Let
\[
        c_u(x,y)
        :=
        \frac1n
        \bigl(
        \Delta_\Sigma u(x)-\langle H(x),y\rangle
        \bigr).
\]
Let $e_1,\ldots,e_n$ be an orthonormal basis of $T_x\Sigma$, let
$E_i(s)$ be its parallel transport along $\gamma(s)=\exp_x(sv)$, and define
\[
        \mathcal R_n^u(x,y,s)
        :=
        \sum_{i=1}^n
        \left\langle
        \overline R(\gamma'(s),E_i(s))E_i(s),
        \gamma'(s)
        \right\rangle.
\]

Then, for {$0<t<r$},
\begin{align}\label{eq:det Bt}
        0\le\det B_t^u
        \le
        \left[1+t c_u(x, y)-\frac{t}{n} \int_0^t\left(1-\frac{s}{t}\right)^2 \mathcal{R}_n^u(x, y, s) d s\right]^n.
\end{align}
\end{proposition}

\begin{proof}
{Fix $0<t<r$.}
For each $i$, let $Y_i$ be the Jacobi field along $\gamma$ satisfying
$$Y_i(0) = e_i, \qquad Y_i(t) =0.$$
Taking the trace in
\eqref{eq:gradient-normal-B-index-form} gives
\begin{align*}
        \operatorname{tr}B_t^u
        =
        t\sum_{i=1}^n I^t(Y_i,Y_i)
        +t\Delta_\Sigma u(x)
        -t\langle H(x),y\rangle.
\end{align*}
Consider the comparison vector fields
\[
        V_i(s):=\left(1-\frac{s}{t}\right)E_i(s).
\]
The two vector fields have the same endpoint values.  Since $\gamma|_{(0,t]}$
has no conjugate point to $x$, the index lemma implies
\[
        I^t(Y_i,Y_i)\le I^t(V_i,V_i).
\]
As $E_i$ is parallel,
\[
        \overline\nabla_sV_i=-\frac1tE_i,
\]
and therefore
\[
I^t(V_i,V_i)
=
\frac1t
-
\int_0^t
\left(1-\frac{s}{t}\right)^2
\left\langle
\overline R(\gamma'(s),E_i(s))E_i(s),
\gamma'(s)
\right\rangle ds.
\]
Summing over $i$ gives
\begin{equation*}
\frac1n\operatorname{tr}B_t^u
\le
1+t c_u(x,y)
-
\frac{t}{n}
\int_0^t
\left(1-\frac{s}{t}\right)^2
\mathcal R_n^u(x,y,s)\,ds.
\end{equation*}
Since $B_t^u$ is positive semidefinite, the arithmetic--geometric mean inequality yields the desired inequality \eqref{eq:det Bt}.
\end{proof}

We next estimate the ambient exponential factor $|\det D (\exp_x)_{tv}|$.

\begin{lemma}
\label{lem:gradient-normal-signed-ambient}
Let $\gamma(s):=\exp_x(sv)$ for $0\leq s\leq r$,
and define
\begin{equation*}
        \mathcal B_{\overline{\operatorname{Ric}}}^{u}(x,y,t)
        :=
        \int_0^t
        s\left(1-\frac{s}{t}\right)
        \overline{\operatorname{Ric}}
        \bigl(\gamma'(s),\gamma'(s)\bigr)\,ds.
\end{equation*}
Then, for every $0<t\le r$,
\begin{equation}
\label{eq:gradient-normal-signed-ambient}
        \left|\det D(\exp_x)_{tv}\right|
        \le
        \exp\left(
        -\mathcal B_{\overline{\operatorname{Ric}}}^{u}(x,y,t)
        \right).
\end{equation}
\end{lemma}

\begin{proof}
If $v=0$, then $D(\exp_x)_0=I$, the curve $\gamma$ is constant, and both
sides of \eqref{eq:gradient-normal-signed-ambient} equal $1$.  We may
therefore assume that
\[
        a:=|v|>0.
\]
By continuity, it suffices to assume $t<r$.
By Lemma~\ref{lem:general-contact-propagation}, the geodesic $\gamma$
remains minimizing on $[0,r]$.  In particular, for every $0<t<r$ there is
no point conjugate to $x$ along $\gamma|_{(0,t]}$.

Choose a parallel orthonormal frame
\[
        E_1(s),\ldots,E_{N-1}(s)
\]
along $\gamma$, orthogonal to $\gamma'(s)$.  Let
\[
        \mathcal J_v(s):v^\perp\longrightarrow\gamma'(s)^\perp
\]
be the Jacobi tensor defined by
\[
        \mathcal J_v(s)w=X_w(s),
\]
where $X_w$ is the Jacobi field along $\gamma$ satisfying
\[
        X_w(0)=0,
        \qquad
        \overline\nabla_sX_w(0)=w.
\]
Thus
\[
        \mathcal J_v(0)=0,
        \qquad
        \mathcal J_v'(0)=I,
\]
and $\mathcal J_v(s)$ is nonsingular for $0<s\le t$.
We compute determinants using parallel oriented orthonormal frames.  Since
\[
        \lim_{s\to0^+}
        \frac{\det\mathcal J_v(s)}{s^{N-1}}=1,
\]
the determinant is positive for small $s>0$.  It cannot change sign
without vanishing, and hence
\[
        \det\mathcal J_v(s)>0
        \qquad\text{for }0<s\le t.
\]

We first relate this tensor to the differential of the exponential map.
Put $\omega:=v/a$.  The radial vector $\omega$ is mapped by
$D(\exp_x)_{sv}$ to $\gamma'(s)/a$, which has unit length.  If
$w\in v^\perp$, then the variation
\[
        (\varepsilon,\tau)
        \longmapsto
        \exp_x\left(\tau\left(v+\frac{\varepsilon}{s}w\right)\right)
\]
shows that
\[
        D(\exp_x)_{sv}(w)
        =
        \frac1s\mathcal J_v(s)w.
\]
Consequently,
\begin{equation}
\label{eq:gradient-normal-point-exp-density}
        \left|\det D(\exp_x)_{sv}\right|
        =
        \frac{\det\mathcal J_v(s)}{s^{N-1}}.
\end{equation}
For each fixed $s\in(0,t]$, define the Jacobi fields
\[
        Y_i^{(s)}(\tau)
        :=
        \mathcal J_v(\tau)\mathcal J_v(s)^{-1}E_i(s),
        \qquad 0\le\tau\le s.
\]
Then
\[
        Y_i^{(s)}(0)=0,
        \qquad
        Y_i^{(s)}(s)=E_i(s),
\]
and the standard first-variation identity for the Jacobi tensor gives
\begin{equation}
\label{eq:gradient-normal-log-J-index}
        \frac{d}{ds}\log\det\mathcal J_v(s)
        =
        \sum_{i=1}^{N-1}
        I^s\bigl(Y_i^{(s)},Y_i^{(s)}\bigr).
\end{equation}
Compare these fields with
\begin{equation}\label{eq:comparison-field}
V_i^{(s)}(\tau) := \frac{\tau}{s}E_i(\tau).
\end{equation}
They have the same endpoint values as $Y_i^{(s)}$. Since there is no conjugate point on
$\gamma|_{(0,s]}$, the index lemma and
\eqref{eq:gradient-normal-log-J-index} imply
\[
        \frac{d}{ds}\log\det\mathcal J_v(s)
        \le
        \sum_{i=1}^{N-1}I^s\bigl(V_i^{(s)},V_i^{(s)}\bigr).
\]
We have
\[
\begin{aligned}
I^s\bigl(V_i^{(s)},V_i^{(s)}\bigr)
={}&
\int_0^s
\left[
\frac1{s^2}
-
\left(\frac{\tau}{s}\right)^2
\left\langle
\overline R\bigl(\gamma'(\tau),E_i(\tau)\bigr)E_i(\tau),
\gamma'(\tau)
\right\rangle
\right]d\tau.
\end{aligned}
\]
Summing over $i$ yields
\[
        \frac{d}{ds}\log\det\mathcal J_v(s)
        \le
        \frac{N-1}{s}
        -
        \int_0^s
        \left(\frac{\tau}{s}\right)^2
        \overline{\operatorname{Ric}}
        \bigl(\gamma'(\tau),\gamma'(\tau)\bigr)
        \,d\tau.
\]
Equivalently,
\[
        \frac{d}{ds}
        \log\left(
        \frac{\det\mathcal J_v(s)}{s^{N-1}}
        \right)
        \le
        -
        \int_0^s
        \left(\frac{\tau}{s}\right)^2
        \overline{\operatorname{Ric}}
        \bigl(\gamma'(\tau),\gamma'(\tau)\bigr)
        \,d\tau.
\]
Integrating from $0$ to $t$, using $\lim _{s \rightarrow 0^{+}} \frac{\det \mathcal{J}_v(s)}{s^{N-1}}=1$,
and changing the order of integration, we obtain
\begin{align*}
\log\left(
        \frac{\det\mathcal J_v(t)}{t^{N-1}}
        \right)
&\le
-
\int_0^t
\int_0^s
\left(\frac{\tau}{s}\right)^2
\overline{\operatorname{Ric}}
\bigl(\gamma'(\tau),\gamma'(\tau)\bigr)
\,d\tau\,ds
\\
&=
-
\int_0^t
\tau\left(1-\frac{\tau}{t}\right)
\overline{\operatorname{Ric}}
\bigl(\gamma'(\tau),\gamma'(\tau)\bigr)
\,d\tau.
\end{align*}
Together with \eqref{eq:gradient-normal-point-exp-density}, this proves
\eqref{eq:gradient-normal-signed-ambient} for $0<t<r$. By continuity, it also holds at $t=r$.

\end{proof}

\begin{theorem}
\label{thm:gradient-normal-signed-main}
Let $(x,y)\in A_r^u$ with $x\in\Sigma\setminus\partial\Sigma$, set
\[
        v:=\nabla^\Sigma u(x)+y,
        \qquad
        \gamma(s):=\exp_x(sv),
\]
and let $c_u$ and $\mathcal R_n^u$ be as in
Proposition~\ref{prop:gradient-normal-main}.  Then, for every $0<t\le r$,
\begin{align*}
\left|\det D\Phi_t^u(x,y)\right|
\le&
 t^m
 \exp\left(
 -\mathcal B_{\overline{\operatorname{Ric}}}^{u}(x,y,t)
 \right)
\notag\\
&\times
\left[
1+t c_u(x,y)
-
\frac{t}{n}
\int_0^t
\left(1-\frac{s}{t}\right)^2
\mathcal R_n^u(x,y,s)\,ds
\right]^n.
\end{align*}
\end{theorem}

\begin{proof}
Combine Proposition~\ref{prop:gradient-normal-main} with
Lemma~\ref{lem:gradient-normal-signed-ambient}.
The estimate at $t=r$ follows by continuity of
$D\Phi_t^u(x,y)$ and the right-hand side.
\end{proof}

\subsection{Tube-volume comparison}

Assume in this subsection that $\Sigma$ has no boundary, and define the unit normal bundle
\[
S^\perp\Sigma
:=
\left\{(x,\omega):x\in\Sigma,\ \omega\in T_x^\perp\Sigma,\ |\omega|=1\right\}.
\]
For $(x,\omega)\in S^\perp\Sigma$, define
\[
        \gamma_{x,\omega}(s):=\exp_x(s\omega)
\]
and the normal cut time
\[
\tau_\Sigma(x,\omega)
:=
\sup\left\{
T\ge0:
\operatorname{dist}\bigl(\gamma_{x,\omega}(s),\Sigma\bigr)=s
\text{ for every }0\le s\le T
\right\}
\in[0,\infty].
\]
We write
\[
        \Phi_t:=\Phi_t^0,
        \qquad
        \Phi_t(x,z)=\exp_x(tz),
\]
and recall that
\[
\Phi_\Sigma:S^\perp\Sigma\times(0,\infty)\longrightarrow M,
\qquad
\Phi_\Sigma(x,\omega,t):=\exp_x(t\omega),
\]
is the normal polar map.  Its Jacobian with respect to
$d\operatorname{vol}_\Sigma\,d\omega\,dt$ is denoted by
$J_\Sigma(x,\omega,t)$.

Let $E_1(s),\ldots,E_n(s)$ be the parallel transports along
$\gamma_{x,\omega}$ of an orthonormal basis of $T_x\Sigma$, and define
\begin{equation}\label{eq:B Ric}
\mathcal B_{\overline{\operatorname{Ric}}}(x,\omega,t)
:=
\mathcal B^0_{\overline{\operatorname{Ric}}}(x,\omega,t)=
\int_0^t
s\left(1-\frac{s}{t}\right)
\overline{\operatorname{Ric}}
\bigl(\gamma_{x,\omega}'(s),\gamma_{x,\omega}'(s)\bigr)\,ds
\end{equation}
and
\begin{equation}\label{eq:R_n}
\mathcal R_n(x,\omega,s)
:=
\mathcal R_n^0(x,\omega,s)=
\sum_{i=1}^n
\left\langle
\overline R\bigl(\gamma_{x,\omega}'(s),E_i(s)\bigr)E_i(s),
\gamma_{x,\omega}'(s)
\right\rangle.
\end{equation}
Equivalently, $\mathcal{R}_n(x,\omega,s)=\overline{\operatorname{Ric}}_n\bigl(\gamma_{x,\omega}'(s),W_s\bigr)$, where $W_s$ is the parallel transport of $T_x\Sigma$ along $\gamma_{x,\omega}$.

\begin{proposition}
For every $(x,\omega)\in S^\perp\Sigma$ and every
$0<t<\tau_\Sigma(x,\omega)$, set
\[
        E_t:=D(\exp_x)_{t\omega},
        \qquad
        L_t:=E_t^{-1}\circ D\Phi_t|_{(x,\omega)}.
\]
Then, with respect to
$T_xM=T_x\Sigma\oplus T_x^\perp\Sigma$,
\[
        L_t=
        \begin{pmatrix}
        B_t&0\\
        C_t&tI_m
        \end{pmatrix},
\]
where $B_t=B_t^0$ is symmetric and positive definite.  In particular,
\begin{enumerate}
    \item
\begin{equation}
\label{eq:det-factorization}
\left|\det D\Phi_t(x,\omega)\right|
=
 t^m\det B_t\,
 \left|\det D(\exp_x)_{t\omega}\right|.
\end{equation}
\item
\begin{align}\label{eq:normal-tangential-factor}
\det B_t
\le
\left[
1-t\langle\sigma(x),\omega\rangle
-
\frac{t}{n}
\int_0^t
\left(1-\frac{s}{t}\right)^2
\mathcal R_n(x,\omega,s)\,ds
\right]^n.
\end{align}
\end{enumerate}
\end{proposition}

\begin{proof}
Choose $r$ with $t<r<\tau_\Sigma(x,\omega)$.  Then
$(x,\omega)\in A_r^0$, and the claim follows
from Proposition~\ref{prop:gradient-normal-decomposition} and  \eqref{eq:det Bt} with
$u\equiv0$ and $y=\omega$.
\end{proof}

\begin{corollary}[Quantitative Heintze--Karcher Jacobian comparison]
\label{cor:quantitative-Heintze-Karcher}
For every $(x,\omega)\in S^\perp\Sigma$ and every
$0<t<\tau_\Sigma(x,\omega)$,
\begin{align}
\label{eq:quantitative-HK-Jacobian}
J_\Sigma(x,\omega,t)
\le{}&
 t^{m-1}
 \exp\left(
 -\mathcal B_{\overline{\operatorname{Ric}}}(x,\omega,t)
 \right)
\notag\\
&\times
\left[
1-t\langle\sigma(x),\omega\rangle
-
\frac{t}{n}
\int_0^t
\left(1-\frac{s}{t}\right)^2
\mathcal R_n(x,\omega,s)\,ds
\right]^n.
\end{align}
\end{corollary}

\begin{proof}
Fix $0<t<\tau_\Sigma(x,\omega)$ and choose
$r$ with
$       t<r<\tau_\Sigma(x,\omega).
$
Then $(x,\omega)\in A_r^0$.  Applying
Theorem~\ref{thm:gradient-normal-signed-main} with
$u\equiv0$ and $y=\omega$ gives
\begin{equation} \label{eq:det DPhi_t^0}
\begin{aligned}
 \left|\det D\Phi_t^0(x,\omega)\right|
\le{}&
t^m
e^{-\mathcal B_{\overline{\operatorname{Ric}}}(x,\omega,t)}
\\
&\times
\left[
1-t\langle\sigma(x),\omega\rangle
-
\frac{t}{n}
\int_0^t
\left(1-\frac{s}{t}\right)^2
\mathcal R_n(x,\omega,s)\,ds
\right]^n.
\end{aligned}
\end{equation}
Recall that $\Phi_t^0:T^\perp\Sigma\to M$ and $\Phi_\Sigma:S^\perp\Sigma\times (0,\infty)\to M$. We have the isometric identification $T_{(x, \omega)}\left(T^{\perp} \Sigma\right) \simeq T_x \Sigma \oplus T_x^{\perp} \Sigma =T_x\Sigma\oplus \omega^\perp \oplus \mathbb R \omega$ and $T_{(x,\omega)}(S^\perp \Sigma\times(0,\infty))\simeq T_x\Sigma\oplus \omega^\perp \oplus \mathbb R\partial_t$. Here $\omega^\perp=\{v\in T_x^\perp\Sigma: v\perp \omega\}$ is the subspace of angular vectors. Note that $D\Phi_t^0|_{(x,\omega)}(\eta)= D\Phi_\Sigma|_{(x,\omega, t)}(\eta)$ for any tangential or angular vector $\eta$. On the other hand, $ D \Phi_{\Sigma}|_{(x, \omega, t)}\left(\partial_t\right)=\gamma_{x, \omega}^{\prime}(t)$ and $D \Phi_t^0|_{(x, \omega)}(0, \omega)=t \gamma_{x, \omega}^{\prime}(t)$. Therefore
\[
        \left|\det D\Phi_t^0(x,\omega)\right|
        =t|\det D\Phi_\Sigma(x,\omega,t)|
        =tJ_\Sigma(x,\omega,t).
\]

In view of \eqref{eq:det DPhi_t^0}, we can obtain \eqref{eq:quantitative-HK-Jacobian}.
\end{proof}

Integrating the pointwise estimate over the minimizing part of the normal
bundle gives the tube-volume comparison.

\begin{corollary}[Quantitative tube-volume comparison]
\label{cor:quantitative-tube-volume}
Suppose that $\Sigma$ is closed, and set
\[
        T_R(\Sigma):=\{q\in M:d(q,\Sigma)<R\}.
\]
Then, for every $R>0$,
\begin{align*}
\operatorname{vol}\bigl(T_R(\Sigma)\bigr)
\le{}&
\int_\Sigma\int_{S_x^\perp\Sigma}
\int_0^{\min\{R,\tau_\Sigma(x,\omega)\}}
 t^{m-1}
 e^{-\mathcal B_{\overline{\operatorname{Ric}}}(x,\omega,t)}
\notag\\
&\qquad\times
\left[
1-t\langle\sigma(x),\omega\rangle
-
\frac{t}{n}
\int_0^t
\left(1-\frac{s}{t}\right)^2
\mathcal R_n(x,\omega,s)\,ds
\right]_+^n
\,dt\,d\omega\,d\operatorname{vol}_\Sigma.
\end{align*}
\end{corollary}

\begin{proof}
The normal polar map covers $T_R(\Sigma)$ up to the normal cut locus.
Its critical values have zero volume by Sard's theorem, while points
admitting more than one minimizing normal representation lie in the
nondifferentiability set of the distance function $d_\Sigma$, which also
has zero volume.  Thus the normal cut locus has zero Riemannian volume.
The area formula, with multiplicity when $\Sigma$ is immersed, gives
\[
\operatorname{vol}\bigl(T_R(\Sigma)\bigr)
\le
\int_\Sigma\int_{S_x^\perp\Sigma}
\int_0^{\min\{R,\tau_\Sigma(x,\omega)\}}
J_\Sigma(x,\omega,t)
\,dt\,d\omega\,d\operatorname{vol}_\Sigma.
\]
Applying Corollary~\ref{cor:quantitative-Heintze-Karcher} then gives the result.
\end{proof}

\begin{corollary}
\label{cor:HK-Ricn-nonnegative}
Suppose that $\overline{\operatorname{Ric}}_n\ge0$.  Then, for every
$0<t<\tau_\Sigma(x,\omega)$,
\[
J_\Sigma(x,\omega,t)
\le
 t^{m-1}
 \left(1-t\langle\sigma(x),\omega\rangle\right)_+^n.
\]
If $\Sigma$ is compact, then
\begin{equation*}
\operatorname{vol}\bigl(T_R(\Sigma)\bigr)
\le
\int_\Sigma\int_{S_x^\perp\Sigma}\int_0^R
 t^{m-1}
 \left(1-t\langle\sigma(x),\omega\rangle\right)_+^n
\,dt\,d\omega\,d\operatorname{vol}_\Sigma.
\end{equation*}
\end{corollary}

\begin{proof}
The hypothesis gives $\mathcal R_n\ge0$
and
$\overline{\operatorname{Ric}}\ge0$.  Thus the two
curvature corrections in
\eqref{eq:quantitative-HK-Jacobian} have the favorable sign and may be
discarded.  Integrating over the normal cut domain and then enlarging the
$t$-integration interval to $[0,R]$ proves the result.
\end{proof}

\subsection{A Riccati estimate for the normal Jacobian}
\label{subsec:direct-riccati}
In this Subsection, we prove another Heintze--Karcher-type volume comparison theorem using a scalar Riccati comparison argument.
Although our main result, Proposition~\ref{prop:direct-riccati-normal-jacobian}, is slightly weaker than Theorem~\ref{thm:gradient-normal-signed-main}, its proof is considerably simpler. We therefore believe that the method may be of independent interest. We will use the notation introduced in Subsection~\ref{subsec:main-factor}.
We begin with a scalar comparison result.

\begin{proposition}[Quantitative scalar Riccati comparison]
\label{prop:quantitative-riccati-comparison}
Let $q_1,q_2:(a,b]\to\mathbb R$ be smooth, let
$\delta:[a,b]\to\mathbb R$ be continuous, and suppose that
\[
q_1'(s)+q_1(s)^2+\delta(s)
\le q_2'(s)+q_2(s)^2
\qquad\text{on }(a,b].
\]
Fix $T\in(a,b]$ and assume that
\[
\lim_{s\to a^+}\bigl(q_2(s)-q_1(s)\bigr)=0,
\qquad
\exp\left(-\int_s^Tq_2(\tau)\,d\tau\right)=O(1)
\quad\text{as }s\to a^+.
\]
Then
\[
q_1(T)
\le q_2(T)
-
\int_a^T
\exp\left(-2\int_r^Tq_2(\tau)\,d\tau\right)
\delta(r)\,dr.
\]
\end{proposition}

\begin{proof}
Set $w:=q_2-q_1$.  The differential inequality gives
\[
w'
\ge q_1^2-q_2^2+\delta
=-2q_2w+w^2+\delta
\ge-2q_2w+\delta.
\]
For $s\in(a,T)$ and $r\in[s,T]$, put
\[
\mu(r):=\exp\left(2\int_s^r q_2(\tau)\,d\tau\right).
\]
Then
\[
\frac{d}{dr}\bigl(\mu(r)w(r)\bigr)\ge\mu(r)\delta(r).
\]
Integrating from $s$ to $T$ and dividing by $\mu(T)$ yields
\[
w(T)
\ge
\exp\left(-2\int_s^Tq_2(\tau)\,d\tau\right)w(s)
+
\int_s^T
\exp\left(-2\int_r^Tq_2(\tau)\,d\tau\right)
\delta(r)\,dr.
\]
The first term on the right tends to zero as $s\to a^+$ by the two
hypotheses.  Letting $s\to a^+$ and recalling that $w=q_2-q_1$ proves the
claim.
\end{proof}

\begin{proposition}
\label{prop:direct-riccati-normal-jacobian}
Assume that $\overline{\operatorname{Ric}} \geq 0$.
Let $(\overline x, \overline y)\in A_r^u$ and set
$$
c:= \frac1n \left(\Delta_\Sigma u(\overline x)-\langle H(\overline x), \overline y\rangle \right).
$$
Suppose that $\mathcal R_n^u(\overline x,\overline y,t) \geq 0$ for all $0 < t < r$.

Then,
$1+ct >0$ for $0<t<r$ and
the function
$$t \longmapsto \frac{|\det D \Phi_t^u(\overline x,\overline y)|}{t^m(1+ct)^n e^{-R_t (\overline x, \overline y)}}$$
is nonincreasing,
where
$$R_t(\overline x, \overline y ) = \mathcal{B}_{\overline{\operatorname{Ric}}}^u(\overline x,\overline y,t) + \int_0^t \frac{(t-\rho)^2}{t(1+ct)} \mathcal{R}_n^u(\overline x,\overline y,\rho) d \rho \geq 0.$$
Moreover, $t\mapsto R_t(\overline x,\overline y)$ is
nondecreasing on $(0,r)$.

In particular, for $0 <t < r$,
$$
|\det D\Phi_t^u(\overline x, \overline y)| \le t^m(1+ct)^n \exp \left( - R_t(\overline x, \overline y)\right).
$$
\end{proposition}

\begin{proof}
Fix $(\overline x, \overline y)\in A_r^u$ and set
$$
\overline z:=\nabla^\Sigma u(\overline x)+\overline y,
\qquad
\overline\gamma(t):=\exp_{\overline x}(t\overline z),
\qquad
0\le t\le r.
$$
Choose an orthonormal basis $\{e_i\}_{i=1}^n$ of $T_{\overline x}\Sigma$
and an orthonormal basis $\{e_\alpha\}_{\alpha=n+1}^{n+m}$ of
$T_{\overline x}^\perp\Sigma$. Let $E_A(t)$ denote the parallel transport of
$e_A$ along $\overline\gamma$.

For $1\le i\le n$, let $X_i(t)$ be the Jacobi field satisfying
$ X_i(0)=e_i, $
$$
\langle \overline\nabla_tX_i(0), e_j\rangle
=
(\nabla^2_\Sigma u)(e_i, e_j)-\langle \mathrm{II}(e_i, e_j), \overline y\rangle,
\qquad 1\le j\le n,
$$
and
$$
\langle \overline\nabla_tX_i(0), e_\beta\rangle
=
\langle \mathrm{II}(e_i, \nabla^\Sigma u), e_\beta\rangle,
\qquad n+1\le \beta\le n+m.
$$
For $n+1\le \alpha\le n+m$, let $X_\alpha(t)$ be the Jacobi field
satisfying
$$
X_\alpha(0)=0,
\qquad
\overline\nabla_tX_\alpha(0)=e_\alpha.
$$

Define $P_{AB}(t):=\langle X_A(t), E_B(t)\rangle$ and $S_{AB}(t):= \langle \overline R(E_A(t), \overline\gamma'(t))\overline\gamma'(t), E_B(t)\rangle$.
Then
$$
P''(t)=-P(t)S(t).
$$
{As in the proof of \cite[Proposition~4.6]{Brendle2023},
$P(t)$ is invertible for $0<t<r$.}
Define $Q(t):=P(t)^{-1}P'(t)$.
Then $Q(t)$ is symmetric and satisfies
$$
Q'(t)+Q(t)^2=-S(t).
$$
Moreover,
$$
P(t)
=
\begin{pmatrix}
I_n+O(t) & O(t)\\
O(t) & tI_m+O(t^2)
\end{pmatrix}
\quad
\text{and}
\quad
Q(t)
=
\begin{pmatrix}
A+O(t) & O(1)\\
O(1) & t^{-1}I_m+O(t)
\end{pmatrix},
$$
where $A_{ij}:= (\nabla^2_\Sigma u)(e_i, e_j)-\langle \mathrm{II}(e_i, e_j), \overline y\rangle$.

Set
$$
q_n(t):=\frac1n\sum_{i=1}^n Q_{ii}(t),
\qquad
q_m(t):=\frac1m\sum_{\alpha=n+1}^{n+m}Q_{\alpha\alpha}(t),
$$
and
$$
\delta_n(t):= \frac1n \sum_{i=1}^n \langle \overline R(E_i(t), \overline\gamma'(t))\overline\gamma'(t), E_i(t)\rangle, \quad
\delta_m(t):= \frac1m \sum_{\alpha=n+1}^{n+m} \langle \overline R(E_\alpha(t), \overline\gamma'(t))\overline\gamma'(t), E_\alpha(t)\rangle.
$$

Taking partial traces in the Riccati equation and using the Cauchy--Schwarz inequality gives
$$
q_n'(t)+q_n(t)^2+\delta_n(t)\le 0 \quad \text{and}\quad q_m'(t)+q_m(t)^2+\delta_m(t)\le 0.
$$
Since
$$
\frac1n\operatorname{tr}A
=
\frac1n
\left(
\Delta_\Sigma u(\overline x)-\langle H(\overline x), \overline y\rangle
\right) = c,
$$
we have
$$q_n(t) = c + O(t), \qquad q_m(t) = \frac{1}{t} + O(t)\quad\text{as }t\to0^+.$$

We first show that $1+ct>0$ for $0\leq t<r$.
Define
$$h(t) : = \exp \left( \int_0^t q_n (\tau ) d \tau \right).$$
Then $h(0) = 1$ and $h'(0) = c$.
Since
$$n\delta_n(t) = \mathcal R_n^u(\overline x,\overline y,t) \geq0,$$
the Riccati inequality for $q_n$ gives
$$h'' = (q_n' + q_n^2 ) h \leq - \delta_n h \leq 0.$$
Hence $h$ is concave, and therefore
$$ 0< h(t) \leq h(0) + t h'(0) = 1 + ct.$$
Thus $1+ct>0$ for $0\leq t<r$.

Define comparison functions
$$
\overline q_n(t):=\frac{c}{1+ct},
\qquad
\overline q_m(t):=\frac1t.
$$
They satisfy
$$
\overline q_n'(t)+\overline q_n(t)^2=0,
\qquad
\overline q_m'(t)+\overline q_m(t)^2=0.
$$
By Proposition~\ref{prop:quantitative-riccati-comparison}, we obtain
$$
q_n(t) \le \overline q_n(t)-E_n(t), \quad \text{where } E_n(t):= \int_0^t \left(\frac{1+c\rho}{1+ct} \right)^2 \delta_n(\rho)\, d\rho,
$$
and
$$
q_m(t) \le \overline q_m(t)-E_m(t), \quad \text{where } E_m(t):= \int_0^t \left(\frac{\rho}{t}\right)^2 \delta_m(\rho)\, d\rho.
$$
Therefore,
\begin{align*}
\frac{d}{dt}\log|\det P(t)|
=\operatorname{tr}Q(t)
=& nq_n(t)+mq_m(t)\\
\le& \frac{nc}{1+ct} + \frac{m}{t} - nE_n(t)-mE_m(t).
\end{align*}
Equivalently,
$$
\frac{d}{dt}
\log
\left(
\frac{|\det P(t)|}{t^m(1+ct)^n}
\right)
\le
-nE_n(t)-mE_m(t).
$$

Set
\begin{equation}\label{eq Rs}
R_t(\overline x, \overline y)
:=
\int_0^t
\{nE_n(\tau)+mE_m(\tau)\}\, d\tau.
\end{equation}
It follows that
$$\frac{d}{dt}
\log
\left(
\frac{|\det P(t)|}
{t^m(1+ct)^n e^{-R_t(\overline x,\overline y)}}
\right)
\leq 0.$$
Thus,
$$t \longmapsto \frac{|\det P(t)|}{t^m(1+ct)^n e^{-R_t (\overline x, \overline y)}}$$
is nonincreasing.
Since
$$\lim_{t \to 0^+}  \frac{|\det P(t)|}{t^m(1+ct)^n e^{-R_t (\overline x, \overline y)}} = 1,$$
we obtain
$$
|\det D\Phi_t^u(\overline x, \overline y)|
=
|\det P(t)|
\le
t^m(1+ct)^n
\exp(-{R_t(\overline x, \overline y)}).
$$

It remains to identify $R_s$. For $0\leq\rho\leq s<r$, set
\begin{equation*}
W_n(s, \rho):= \frac{(s-\rho)(1+c\rho)}{1+cs}, \quad W_m(s, \rho):= \rho\left(1-\frac{\rho}{s}\right)
\end{equation*}
Then
$$W_n(s,\rho)-W_m(s,\rho)
    =
    \frac{(s-\rho)^2}{s(1+cs)}
    \geq 0. $$

By Fubini's theorem, \eqref{eq Rs} becomes
\begin{equation}\label{eq:Rs2}
R_s = \int_0^s W_m(s, \rho)\left[n\delta_n(\rho)+m\delta_m(\rho) \right]\, d\rho + \int_0^s \left[W_n(s, \rho)-W_m(s, \rho) \right] n\delta_n(\rho)\, d\rho.
\end{equation}
Since
$$
n\delta_n(\rho)+m\delta_m(\rho) = \overline{\operatorname{Ric}}
\bigl(\overline\gamma'(\rho),\overline\gamma'(\rho)\bigr)
$$
and
$$
n\delta_n(\rho) = \mathcal R_n^u(\overline x,\overline y,\rho),
$$
the definitions of
$\mathcal B_{\overline{\operatorname{Ric}}}^u$
and $\mathcal R_n^u$ give
\begin{equation*}
    R_s(\overline x,\overline y)
=
\mathcal B_{\overline{\operatorname{Ric}}}^u(\overline x,\overline y,s)
+
\int_0^s
\frac{(s-\rho)^2}{s(1+cs)}
\mathcal R_n^u(\overline x,\overline y,\rho)\,d\rho.
\end{equation*}

Since $W_m(s,\rho)\ge0$, $\overline{\operatorname{Ric}}\ge0$, and $\mathcal R_n^u\ge0$, we conclude that
$$R_s(\overline x,\overline y)\ge0.$$

We now prove that $R_t(\overline x,\overline y)$ is nondecreasing in $t$.
By \eqref{eq Rs} and the definitions of $E_n$ and $E_m$,
\begin{align*}
\frac{d}{dt}R_t(\overline x,\overline y)
&=nE_n(t)+mE_m(t)\\
&=
\int_0^t
\left(\frac{\rho}{t}\right)^2
\overline{\operatorname{Ric}}
\bigl(\overline\gamma'(\rho),\overline\gamma'(\rho)\bigr)\,d\rho\\
&\quad+
\int_0^t
\left[
\left(\frac{1+c\rho}{1+ct}\right)^2
-
\left(\frac{\rho}{t}\right)^2
\right]
\mathcal R_n^u(\overline x,\overline y,\rho)\,d\rho.
\end{align*}
For $0\le\rho\le t<r$, the positivity of $1+ct$ gives $ \frac{1+c\rho}{1+ct}-\frac{\rho}{t} = \frac{t-\rho}{t(1+ct)} \ge0 $.
Both fractions are nonnegative, so $\left(\frac{1+c \rho}{1+c t}\right)^2-\left(\frac{\rho}{t}\right)^2\ge0$. Since $\overline{\operatorname{Ric}}\ge0$ and $\mathcal R_n^u\ge0$, both integrals above are nonnegative. Thus $\frac{d}{dt}R_t(\overline x,\overline y)\ge0$.
\end{proof}

\section{A quantitative Fenchel--Willmore inequality}
\label{sec:quantitative-Willmore}

In this section, we prove Theorem~\ref{thm:intro-quantitative-Willmore}, the quantitative Fenchel--Willmore inequality. The main tool is Corollary~\ref{cor:quantitative-tube-volume}, which, under the assumption $\overline{\operatorname{Ric}}_n\ge0$, provides the required volume estimate for tubular neighborhoods of $\Sigma$. If we neglect the remainder terms, this approach yields a simpler proof of the Fenchel--Willmore inequality established by Ji and Kwong~\cite{JiKwong}. The additional difficulty in the present quantitative setting is to retain and estimate the nontrivial curvature remainder terms, which have no counterpart in their result.

Set $N:=n+m$. Throughout this section, let $(M^N,\overline g)$ be a complete noncompact Riemannian manifold satisfying $\overline{\operatorname{Ric}}_n\ge0$, and let $\Sigma^n\subset M^N$ be a closed (compact and without boundary) immersed submanifold.  We write $H=n\sigma$.  As observed in the proof of Corollary~\ref{cor:HK-Ricn-nonnegative}, the intermediate-Ricci hypothesis implies $\overline{\operatorname{Ric}}\ge0$.  Hence the asymptotic volume ratio
\[
\theta
:=
\lim_{R\to\infty}
\frac{\operatorname{vol}(B_o(R))}{|\mathbb B^N|R^N}
\in[0,1]
\]
exists and is independent of the base point $o\in M$. { No positivity assumption on $\theta$ is needed below.}

Let
$S_x^\perp\Sigma$ be the unit normal space of $\Sigma$ at $x$.
For $z\in S_x^\perp\Sigma$, let
\[
\gamma_{x,z}(s):=\exp_x(sz)
\]
be the unit-speed normal geodesic.  For $R>0$, define
\begin{equation}\label{def:B Ric}
\mathcal B_{\overline{\operatorname{Ric}},R}(x,z)
:=\mathcal{B}_{\overline{\operatorname{Ric}}}(x, z, R)= \int_0^R s\left(1-\frac{s}{R}\right) \overline{\operatorname{Ric}} \bigl(\gamma_{x,z}'(s),\gamma_{x,z}'(s)\bigr) \,ds,
\end{equation}
where $\mathcal{B}_{\overline{\operatorname{Ric}}}$ is defined by \eqref{eq:B Ric},
and
\[
\mathcal H_{n,R}(x,z)
:=
\frac1n
\int_0^R
\left(1-\frac{s}{R}\right)^2
\mathcal R_n(x,z,s)
\,ds,
\]
where $\mathcal R_n$ is defined by \eqref{eq:R_n}.
Both quantities are nonnegative.
Moreover, differentiating \eqref{def:B Ric} with respect to $R$ shows that
$\mathcal B_{\overline{\operatorname{Ric}},R}(x,z)$ is nondecreasing in $R$ if $\overline{\operatorname{Ric}}_n \geq 0$.
We may therefore define
\[
\mathcal B_{\overline{\operatorname{Ric}}}(x,z)
:= \lim_{R\to\infty} \mathcal B_{\overline{\operatorname{Ric}},R}(x,z) \in[0,\infty].
\]
Finally, for $z\in S_x^\perp \Sigma$, set
\begin{equation*}
    \mathcal T_n(x,z)
:= \liminf_{R\to\infty} e^{-\mathcal B_{\overline{\operatorname{Ric}},R}(x,z)} \min\left\{ \mathcal H_{n,R}(x,z), \langle -\sigma(x),z\rangle \right\}.
\end{equation*}

\subsection{Quantitative Fenchel--Willmore inequality}
We are now in a position to state the main result of this section, namely, the quantitative Fenchel--Willmore inequality.
\begin{theorem}[Quantitative Fenchel--Willmore inequality]
\label{thm:quantitative-Willmore}
Let $(M^N,\overline g)$ be a complete noncompact Riemannian manifold satisfying
$\overline{\operatorname{Ric}}_n\ge0$, and let
$\Sigma^n\subset M^N$ be a closed
immersed submanifold.  Then
\begin{equation}\label{eq:quantitative-Willmore}
\begin{aligned}
\int_{\Sigma}|\sigma|^n d \operatorname{vol}_{\Sigma}-\theta\left|\mathbb{S}^n\right| \geq & \frac{1}{C_{n, m}} \int_{\Sigma} \int_{S_x^\perp\Sigma}\langle-\sigma(x), z\rangle^n_{{+}}\left(1-e^{-\mathcal{B}_{\overline {\operatorname{Ric}}}(x, z)}\right) d S_z d \operatorname{vol}_{\Sigma} \\
& +\frac{1}{C_{n, m}} \int_{\Sigma} \int_{S_x^\perp\Sigma}\langle-\sigma(x), z\rangle^{n-1}_{{{+}}}\mathcal{T}_n(x, z) d S_z d \operatorname{vol}_{\Sigma},
\end{aligned}
\end{equation}
where $C_{n,m} := \frac{|\mathbb S^{N-1}|}{|\mathbb S^n|}$.
\end{theorem}

\begin{lemma}
\label{lem:quantitative-binomial}
For $a,b\ge0$, $0\le q\le1$, and every integer $n\ge1$,
\[
q(a-b)_+^n
\le
 a^n-a^n(1-q)-q a^{n-1}\min\{a,b\}.
\]
\end{lemma}

\begin{proof}
The right-hand side equals
$q a^{n-1}\bigl(a-\min\{a,b\}\bigr)$.
If $b\ge a$, both sides vanish.  If $0\le b<a$, then
\[
(a-b)^n\le a^{n-1}(a-b),
\]
which proves the claim.
\end{proof}

\begin{lemma} \label{lem:spherical-beta-integral}
Let $w\in\mathbb R^m$.
Then for every integer $n\ge1$,
$$\int_{\mathbb S^{m-1}} \langle w,z\rangle^n_{{{+}}}\,dS_z = \frac{|\mathbb S^{N-1}|}{|\mathbb S^n|}|w|^n.$$
\end{lemma}
\begin{proof}
By \cite[Lemma 2.5]{Brendle-ABP}, $\int_{\mathbb B^m(r)} \langle w, z\rangle_{+}^n dz=r^{N}\frac{\left|\mathbb{B}^{N}\right|}{\left|\mathbb{S}^n\right|}|w|^n$, where  $\mathbb B^m(r)$ is the ball of radius $r$ in $\mathbb R^m$. Differentiating this identity at $r=1$ and Fubini's theorem then gives the result, as $|\mathbb S^{N-1}|=N|\mathbb B^{N}|$.
\end{proof}

\begin{proof}[Proof of Theorem~\ref{thm:quantitative-Willmore}]
Fix $o\in M$ and set $R_0:=\max_{x\in\Sigma}d(o, x)$. Since $B_o(R-R_0)\subset T_R(\Sigma)\subset B_o(R+R_0)$ for all sufficiently large $R$,
it is easy to see that
\begin{equation}
\label{eq:tube-asymptotic-volume}
\lim_{R\to\infty}
\frac{\operatorname{vol}(T_R(\Sigma))}
{|\mathbb B^N|R^N}
=\theta.
\end{equation}
For $t>0$, suppressing the dependence on $(x, \omega)$, define
$$
a_t:=\left(1-t\langle\sigma(x), \omega\rangle\right)_+,
\qquad
b_t:=t\mathcal H_{n, t}(x, \omega),
\qquad
q_t:=e^{-\mathcal B_{\overline{\operatorname{Ric}}, t}(x, \omega)}.
$$
The hypothesis $\overline{\operatorname{Ric}}_n\ge0$ implies
$b_t\ge0$ and $0\le q_t\le1$. Moreover,
$$
\left[ 1-t\langle\sigma(x), \omega\rangle -t\mathcal H_{n, t}(x, \omega) \right]_+^n =(a_t-b_t)_+^n.
$$
By Lemma~\ref{lem:quantitative-binomial},
\begin{equation} \label{eq: q a b}
q_t(a_t-b_t)_+^n
\le a_t^n-a_t^n(1-q_t) -q_ta_t^{n-1}\min\{a_t, b_t\}.
\end{equation}
Corollary \ref{cor:quantitative-tube-volume} and \eqref{eq: q a b} then imply
\begin{equation}
\begin{aligned}\label{eq:vol TR estimate}
& \operatorname{vol}(T_R(\Sigma))\\
\le& \int_{\Sigma} \int_{S_x^{\perp} \Sigma} \int_0^{R} t^{m-1} e^{-\mathcal{B}_{\overline{\operatorname{Ric}}}(x, \omega, t)}
\left[1-t\langle\sigma(x), \omega\rangle-\frac{t}{n} \int_0^t\left(1-\frac{s}{t}\right)^2 \mathcal{R}_n(x, \omega, s) d s\right]_{+}^n d t d \omega d \operatorname{vol}_{\Sigma}\\
\le& \mathcal M_R-\mathcal E_R^{\overline{\operatorname{Ric}}} -\mathcal E_R^n,
\end{aligned}
\end{equation}
where
\begin{align*}
\mathcal M_R
&:= \int_\Sigma\int_{S_x^\perp\Sigma}\int_0^R t^{m-1}a_t^n \, dt\, dS_\omega\, d\operatorname{vol}_\Sigma, \\ \mathcal E_R^{\overline{\operatorname{Ric}}}
&:= \int_\Sigma\int_{S_x^\perp\Sigma}\int_0^R t^{m-1}a_t^n(1-q_t) \, dt\, dS_\omega\, d\operatorname{vol}_\Sigma, \\ \mathcal E_R^n
&:= \int_\Sigma\int_{S_x^\perp\Sigma}\int_0^R t^{m-1}q_ta_t^{n-1}\min\{a_t, b_t\} \, dt\, dS_\omega\, d\operatorname{vol}_\Sigma.
\end{align*}
Let
\begin{align*}
\mathcal I_{\overline{\operatorname{Ric}}}
&:= \int_{\Sigma} \int_{S_x^\perp\Sigma} \langle -\sigma(x),z\rangle ^n_{{{+}}} \left(1-e^{-\mathcal B_{\overline{\operatorname{Ric}}}(x, z)}\right) \, dS_z\, d\operatorname{vol}_\Sigma, \\
\mathcal I_n
&:= \int_{\Sigma} \int_{S_x^\perp\Sigma} \langle-\sigma(x), z\rangle^{n-1}_{{{+}}}\mathcal T_n(x, z) \, dS_z\, d\operatorname{vol}_\Sigma.
\end{align*}

We claim that
\begin{align}
\lim_{R\to\infty}\frac{\mathcal M_R}{R^N}
& = \frac{C_{n, m}}{N} \int_\Sigma|\sigma|^n\, d\operatorname{vol}_\Sigma,
\label{eq:direct-main-limit}\\
\lim_{R\to\infty} \frac{\mathcal E_R^{\overline{\operatorname{Ric}}}}{R^N}
& = \frac{\mathcal I_{\overline{\operatorname{Ric}}}}{N},
\label{eq:direct-Ricci-limit}\\
\liminf_{R\to\infty}\frac{\mathcal E_R^n}{R^N}
& \ge \frac{\mathcal I_n}{N}.
\label{eq:direct-tangential-limit}
\end{align}
To see this, let $t=R\rho$, then we have
$$
\frac{\mathcal M_R}{R^N}
= \int_\Sigma\int_{S_x^\perp\Sigma}\int_0^1 \rho^{m-1} \left(\frac{a_{R\rho}(x, \omega)}{R}\right)^n \, d\rho\, dS_\omega\, d\operatorname{vol}_\Sigma.
$$
Moreover,
$$
\frac{a_{R\rho}(x, \omega)}{R} \longrightarrow \rho\left(-\langle\sigma(x), \omega\rangle\right)_+,
\qquad
0\le \frac{a_{R\rho}(x, \omega)}{R}
\le \frac1R+|\sigma(x)|.
$$
Dominated convergence theorem and Lemma~\ref{lem:spherical-beta-integral} therefore give
\begin{align*}
\lim_{R\to\infty}\frac{\mathcal M_R}{R^N}
& = \int_\Sigma\int_{S_x^\perp\Sigma}\int_0^1 \rho^{N-1} \left(-\langle\sigma(x), \omega\rangle\right)_+^n \, d\rho\, dS_\omega\, d\operatorname{vol}_\Sigma\\
& = \frac1N \int_{\Sigma} \int_{S_x^\perp\Sigma} \left\langle -{\sigma(x)}, z\right\rangle  ^n_{{{+}}}\, dS_z\, d\operatorname{vol}_\Sigma\\
& = \frac{C_{n, m}}{N} \int_\Sigma|\sigma|^n\, d\operatorname{vol}_\Sigma,
\end{align*}
which proves \eqref{eq:direct-main-limit}. Since $q_{R\rho}(x, \omega) \longrightarrow e^{-\mathcal B_{\overline{\operatorname{Ric}}}(x, \omega)}$
for every $\rho>0$, the same argument proves
\eqref{eq:direct-Ricci-limit}.

For the remaining term, fix $x\in\Sigma$ and $z\in \{\zeta \in S_x^\perp\Sigma: \langle -\sigma(x),\zeta \rangle \ge0\}$.
For $t=R\rho$,
$$
a_t=t\left(\langle -\sigma(x), z\rangle+\frac1t\right), \qquad b_t=t\mathcal H_{n, t}(x, z),
$$
and hence
\begin{align*}
\frac{q_{R\rho}a_{R\rho}^{n-1}\min\{a_{R\rho},b_{R\rho}\}}{R^n}
&=\rho^n q_{R\rho}
\left(\langle-\sigma(x),z\rangle+\frac1{R\rho}\right)^{n-1}\\
&\quad\times
\min\left\{\langle-\sigma(x),z\rangle+\frac1{R\rho},
\mathcal H_{n,R\rho}(x,z)\right\}.
\end{align*}
Because
$$
\left| \min\left\{\langle -\sigma(x), z\rangle+\frac1t, \mathcal H_{n, t}\right\} -\min\left\{\langle -\sigma(x), z\rangle, \mathcal H_{n, t}\right\} \right|
\le\frac1t,
$$
the definition of $\mathcal T_n$  gives
$$
\liminf_{R\to\infty} \frac{q_{R\rho}a_{R\rho}^{n-1} \min\{a_{R\rho}, b_{R\rho}\}}{R^n}
\ge \rho^n {\langle -\sigma (x) , z \rangle_+^{n-1}} \mathcal T_n(x, z).
$$
Fatou's lemma applied to $\frac{\mathcal{E}_R^n}{R^N}$, followed by
$\int_0^1\rho^{N-1}\, d\rho=1/N$, proves
\eqref{eq:direct-tangential-limit}.

By dividing \eqref{eq:vol TR estimate} by $R^N$ and then taking the liminf, and using \eqref{eq:tube-asymptotic-volume} together with \eqref{eq:direct-main-limit}--\eqref{eq:direct-tangential-limit}, we arrive at
$$
\theta|\mathbb B^N|
\le \frac1N \left[ C_{n, m}\int_\Sigma|\sigma|^n\, d\operatorname{vol}_\Sigma -\mathcal I_{\overline{\operatorname{Ric}}} -\mathcal I_n \right].
$$
Finally,
$$
N|\mathbb B^N| = |\mathbb S^{N-1}| = C_{n, m}|\mathbb S^n|.
$$
Rearranging the preceding inequality and dividing by $C_{n, m}$ gives
$$
\int_\Sigma|\sigma|^n\, d\operatorname{vol}_\Sigma
-\theta|\mathbb S^n|
\ge
\frac{\mathcal I_{\overline{\operatorname{Ric}}}}{C_{n, m}}
+
\frac{\mathcal I_n}{C_{n, m}},
$$
which is precisely \eqref{eq:quantitative-Willmore}.
\end{proof}

{The two curvature corrections are nonnegative. Discarding them gives the following standard Fenchel--Willmore inequality.}

\begin{corollary}
Under the assumptions of Theorem~\ref{thm:quantitative-Willmore},
\begin{equation}\label{eq:standard-willmore}
\int_\Sigma|\sigma|^n\,d\operatorname{vol}_\Sigma \ge \theta|\mathbb S^n|.
\end{equation}

\end{corollary}
\subsection{Equality case of Theorem \ref{thm:quantitative-Willmore}}
We now investigate the equality case of Theorem \ref{thm:quantitative-Willmore}.

Let us first define the notion of active directions.
For $(x, z)\in S^\perp\Sigma$, set
$$
h(x, z):=\lim_{R\to\infty}\mathcal H_{n, R}(x, z),
$$
and call $(x, z)$ \emph{active} if
$$
\langle-\sigma(x),z\rangle>0, \qquad
\mathcal B_{\overline{\operatorname{Ric}}}(x, z)<\infty,
\qquad
h(x, z)<\langle-\sigma(x),z\rangle.
$$
Let $\mathcal A\subset S^\perp\Sigma$ be the set of active directions and set
$$
C_{\mathcal A}(R)
:=
\{(x, sz):(x, z)\in\mathcal A, \ 0<s<R\}.
$$
We call a geodesic ray active if it is of the form $\gamma_{x,z}:[0,\infty)\to M$ for some $(x,z)\in \mathcal A$.

\begin{theorem}
\label{thm:metric-rigidity}
Assume the hypotheses of
Theorem~\ref{thm:quantitative-Willmore}.
Then equality holds in
\eqref{eq:quantitative-Willmore} if and only if, after removing a null
subset from $\mathcal A$, the following properties hold.

\begin{enumerate}

\item
With respect to the horizontal--vertical splitting
determined by the normal connection,
\begin{equation}
\label{eq:active-conical-metric}
\Phi_{\Sigma}^*\overline g
=
ds^2+s^2g_{\mathcal V}
+\left(1+s\langle-\sigma(x),z\rangle\right)^2\pi^*g_\Sigma
\end{equation}
on $C_{\mathcal A}(\infty)$. Here $\Phi_{\Sigma} :S^\perp \Sigma\times (0,\infty)\to M$ is the normal exponential map defined by  $\Phi_{\Sigma}(x,z,s)=\exp_x(sz)$, and $g_{\mathcal V}$ is the metric on the fiber of $\pi:S^\perp \Sigma\to \Sigma$.
\item
The normal exponential map is injective on
$C_{\mathcal A}(\infty)$.
\item
If $\theta>0$,
$$
\lim_{R\to\infty}
\frac{\operatorname{vol}\left(\Phi_{\Sigma} C_{\mathcal A}(R)\right)}
{\operatorname{vol}(T_R(\Sigma))}
=1.
$$
\end{enumerate}

\end{theorem}
\begin{proof}
Assume first that equality holds in \eqref{eq:quantitative-Willmore}. We retain the notation from the proof of Theorem~\ref{thm:quantitative-Willmore}. Since $\mathcal H_{n, R}(x, z)$ is nondecreasing in $R$, the liminf in \eqref{eq:direct-tangential-limit} is in fact a limit.

From the proof of Theorem \ref{thm:quantitative-Willmore}, we see that
\begin{equation}
\label{eq:rigidity-total-loss}
\lim_{R\to\infty} \frac{ \mathcal M_R -\mathcal E_R^{\overline{\operatorname{Ric}}} -\mathcal E_R^n -\operatorname{vol}(T_R(\Sigma)) }{R^N}
=0.
\end{equation}

Since we replace $\tau_\Sigma(x,z)$ by $R$ when we apply Corollary \ref{cor:quantitative-tube-volume}, we have
\begin{align}
0
= \lim_{R\to\infty}\frac1{R^N} \int_\Sigma\int_{S_x^\perp\Sigma} \int_{\min\{R, \tau_\Sigma(x, z)\}}^R t^{m-1}q_t(a_t-b_t)_+^n \, dt\, dS_z\, d\operatorname{vol}_\Sigma.
\label{eq:vanishing-cut-loss}
\end{align}
Let $(x, z)\in \mathcal A$ and suppose that
$\tau_\Sigma(x, z)<\infty$. Letting $t=R\rho$ for $\rho\in(0, 1)$, we have
$$
q_{R\rho} \longrightarrow e^{-\mathcal B_{\overline{\operatorname{Ric}}}(x, z)}>0
\quad \text{ and }\quad
\frac{(a_{R\rho}-b_{R\rho})_+}{R} \longrightarrow \rho \left[ \langle-\sigma(x), z\rangle - \lim_{t\to\infty}\mathcal H_{n, t}(x, z) \right]>0
$$
as $R\to \infty$.
\eqref{eq:vanishing-cut-loss} then shows that the set of active directions with finite $\tau_\Sigma$ is of measure zero.
After removing this set, every active normal ray satisfies
$\tau_\Sigma(x, z)=\infty$.

Along an active ray, set
$$\phi_{x,z}(s) := e^{\mathcal B_{\overline{\operatorname{Ric}},s}(x,z)} \left|\det D(\exp_x)_{sz}\right|.$$
It is easy to see that $\phi_{x,z}(s)\to1$ as $s\downarrow0$. We claim that $\phi_{x,z}$ is nonincreasing. Indeed, direct differentiation gives
$$
\frac{d}{d s} \mathcal{B}_{\overline{\operatorname{Ric}}, s}(x, z)=\frac{1}{s^2} \int_0^s \tau^2 \overline{\operatorname{Ric}}\left(\gamma^{\prime}(\tau), \gamma^{\prime}(\tau)\right) d \tau
$$
whereas the proof of Lemma \ref{lem:gradient-normal-signed-ambient} shows that
$$
\frac{d}{d s} \log \left|\det D\left(\exp _x\right)_{s z}\right| \leq-\frac{1}{s^2} \int_0^s \tau^2 \overline{\operatorname{Ric}}\left(\gamma^{\prime}(\tau), \gamma^{\prime}(\tau)\right) d \tau.
$$
This implies $\phi_{x, z}$ is nonincreasing.

We claim that $\lim _{s \rightarrow \infty} \phi_{x, z}(s)=1$ for almost every active direction. Suppose not, if $\lim_{s\to\infty}\phi_{x,z}(s)<1$ on a positive-measure set of active directions, then monotonicity of $\phi_{x,z}$ and continuity implies that there exists $c<1$, $\delta>0$, $s_0>0$, and a positive-measure subset $E \subset \mathcal{A}$, such that for $(x,z)\in E$ and $s\ge s_0$,
$$
\phi_{x, z}(s) \leq c, \quad e^{-\mathcal{B}_{\overline {\operatorname{Ric}}}(x, z)} \geq \delta, \text{ and }\quad\langle-\sigma(x), z\rangle-h(x, z) \geq \delta .
$$
Let $U_s:=s^{m-1} q_s\left(a_s-b_s\right)_{+}^n$, which is the integrand in \eqref{eq:vol TR estimate}. Since $\mathcal{H}_{n, s}(x, z) \leq h(x, z)$ and $q_s \geq e^{-\mathcal{B}_{\overline{\operatorname{Ric}}}(x, z)}$, we have
$$
U_s \geq \delta^{n+1} s^{N-1} \quad \text { on } E \times\left[s_0, \infty\right)
$$
and so \eqref{eq:det-factorization} and \eqref{eq:normal-tangential-factor} give $J_\Sigma(x,z,s)\le cU_s$ there.

By \eqref{eq: q a b}, the integrand $I_s:=s^{m-1}\left(a_s^n-a_s^n\left(1-q_s\right)-q_s a_s^{n-1} \min \left\{a_s, b_s\right\}\right)$ defining
$\mathcal M_R-\mathcal E_R^{\overline{\operatorname{Ric}}}-\mathcal E_R^n$ satisfies $I_s\ge U_s$. Therefore, if we let
$d \mu(x, z):=d S_z d \operatorname{vol}_{\Sigma}$ and $ |E|:=\mu(E)$, then for $R>s_0$,
\begin{equation*}
\begin{aligned}
\mathcal{M}_R  -\mathcal{E}_R^{\overline{\operatorname{Ric}}}-\mathcal{E}_R^n-\operatorname{vol}\left(T_R(\Sigma)\right)
& \geq \int_E \int_{s_0}^R\left(U_s-J_{\Sigma}(x, z, s)\right) d s d \mu(x, z) \\
& \geq \frac{(1-c) \delta^{n+1}|E|}{N}\left(R^N-s_0^N\right)
\end{aligned}
\end{equation*}
Dividing this by $R^N$ and letting $R\to \infty$ implies \eqref{eq:rigidity-total-loss} is a strict inequality, a contradiction. Hence
$\lim_{s\to\infty}\phi_{x,z}(s)=1$ for almost every active direction and monotonicity then yields
$\phi_{x,z}(s)\equiv1$.
In other words,
\begin{equation}\label{eq:ambient-exp-equality}
e^{\mathcal B_{\overline{\operatorname{Ric}},s}(x,z)} \left|\det D(\exp_x)_{sz}\right| =1
\qquad\text{for every }s>0.
\end{equation}

We now use the equality case in the proof of
Lemma~\ref{lem:gradient-normal-signed-ambient}. Let $P_s$ denote
parallel transport along $\gamma_{x, z}$ and let
$\mathcal J_{x, z}$ be the Jacobi tensor. From its proof, we see that for fixed $s>0$, the
Jacobi field
$$
Y_E^{(s)}(\tau) := \mathcal J_{x, z}(\tau) \mathcal J_{x, z}(s)^{-1}P_sE
$$
is equal to
$$
V_E^{(s)}(\tau) := \frac{\tau}{s}P_\tau E.
$$

Differentiating this identity at $\tau=0$ gives
$$\mathcal J_{x, z}(s)^{-1}P_sE=\frac1sE, $$
and hence
\begin{equation}
\label{eq:point-Jacobi-rigidity}
\mathcal J_{x, z}(s)E=sP_sE
\qquad\text{for every }E\perp z.
\end{equation}

The Jacobi equation applied to
\eqref{eq:point-Jacobi-rigidity} gives
\begin{equation}\label{eq:raidal-curvature-identity}
\overline R({P_s}E, \gamma_{x, z}'(s))\gamma_{x, z}'(s)=0 \qquad\text{for every }E\perp\gamma_{x, z}'.
\end{equation}

From \eqref{eq:point-Jacobi-rigidity} and \eqref{eq:gradient-normal-point-exp-density}, we have
$\left|\det D(\exp_x)_{sz}\right|=1$, so
\eqref{eq:ambient-exp-equality} gives
$$\mathcal B_{\overline{\operatorname{Ric}}, s}(x, z)=0. $$
The identity \eqref{eq:raidal-curvature-identity} also gives
$$ \mathcal R_n(x, z, s)=0 \quad\text{and} \quad \mathcal H_{n, s}(x, z)=0.$$
It remains to determine the tangential Jacobi fields. Equations \eqref{eq:gradient-normal-det-factorization}, \eqref{eq:gradient-normal-B-index-form}, \eqref{eq:raidal-curvature-identity}, and the derivation of \eqref{eq:quantitative-HK-Jacobian} imply that the normal polar Jacobian along the ray is
$$J_\Sigma(x,z,s) = s^{m-1}\det(I+sA_z).$$
Here we use the convention
$A_zX=(\overline\nabla_Xz)^\top$.
Since the ray minimizes for all time, it has no focal point. Thus
$I+sA_z$ is nonsingular for every $s>0$, and hence
$A_z\ge0$.
Moreover, the arithmetic--geometric mean inequality gives
$$\det A_z \le \langle-\sigma(x),z\rangle^n,$$
with equality if and only if
$A_z=\langle-\sigma(x),z\rangle I$

If this determinant inequality were strict on a positive-measure set
of active directions, then the actual Jacobian would have a strictly
smaller leading coefficient than the upper Jacobian in
\eqref{eq:quantitative-HK-Jacobian}. Similar to the argument of \eqref{eq:ambient-exp-equality}, this would produce a positive
$R^N$-order loss in \eqref{eq:vol TR estimate}, contradicting
\eqref{eq:rigidity-total-loss}. Therefore,
\begin{equation*}
A_z=\langle-\sigma(x),z\rangle I
\end{equation*}
for almost every active direction.

Let $X\in T_x\Sigma$. The horizontal Jacobi field of the normal exponential map has initial conditions
$$J_X(0)=X, \qquad J_X'(0)=A_zX = \langle-\sigma(x),z\rangle X.$$
By \eqref{eq:raidal-curvature-identity} and the Jacobi equation, together with the above initial conditions,
$$J_X(s) = \left( 1+s\langle-\sigma(x),z\rangle \right)P_sX.$$
Likewise, if $\eta\in T_x^\perp\Sigma\cap z^\perp$, the angular Jacobi
field satisfies
$$
J_\eta(0)=0, \qquad J_\eta'(0)=\eta,
$$
and therefore
$$J_\eta(s)=sP_s\eta. $$
These Jacobi fields and the radial field are mutually orthogonal.
Consequently, with respect to the horizontal--vertical splitting,
$$
(\Phi_{\Sigma})^*\overline g = ds^2+s^2g_{\mathcal V} + \left( 1+s\langle-\sigma(x), z\rangle \right)^2\pi^*g_\Sigma,
$$
which is \eqref{eq:active-conical-metric}.

Two distinct active rays cannot meet. Indeed, if they first met, the
two minimizing segments would form a corner, contradicting the fact
that either ray remains minimizing beyond the meeting point. Thus the
normal exponential map is injective on $C_{\mathcal A}(\infty)$.

On the active set we have$$
\mathcal B_{\overline{\operatorname{Ric}}}(x, z)=0,
\qquad
\lim_{t\to\infty}\mathcal H_{n, t}(x, z)=0.
$$
Subtracting the two curvature remainders in \eqref{eq:quantitative-Willmore} from the identity in
Lemma~\ref{lem:spherical-beta-integral}, equality gives
\begin{equation*}
\int_{\mathcal A}
\langle-\sigma(x), z\rangle^n
\, dS_z\, d\operatorname{vol}_\Sigma
=
N|\mathbb B^N|\theta.
\end{equation*}
Using the Jacobian determined by the metric, we obtain
\begin{align*}
\operatorname{vol}\left(\Phi_{\Sigma} C_{\mathcal A}(R)\right)
& =
\int_{\mathcal A}\int_0^R
s^{m-1}
\left(
1+s\langle-\sigma(x), z\rangle
\right)^n
\, ds\, dS_z\, d\operatorname{vol}_\Sigma\\
& =
\theta|\mathbb B^N|R^N+O(R^{N-1}).
\end{align*}
If $\theta>0$, together with \eqref{eq:tube-asymptotic-volume}, this proves
$$
\lim_{R\to\infty}
\frac{\operatorname{vol}\left(\Phi_{\Sigma} C_{\mathcal A}(R)\right)}
{\operatorname{vol}(T_R(\Sigma))}
=1.
$$

Conversely, suppose that the two properties in the statement hold.
The metric formula implies that its Jacobian is
$$
J_\Sigma(x,z,s)=s^{m-1} \left( 1+s\langle-\sigma(x), z\rangle \right)^n.
$$
It is not hard to then deduce that \eqref{eq:quantitative-Willmore} is an equality.

\end{proof}

\begin{example}\label{ex:equality}
Equality in Theorem~\ref{thm:quantitative-Willmore} need not imply equality in the standard Fenchel--Willmore inequality \eqref{eq:standard-willmore}.

Let $M^{n+2}=\mathbb S^{n+1}\times\mathbb R$ with the product metric. Then $\overline{\operatorname{Ric}}_n\ge0$, while the volume growth is linear and hence $\theta=0$. Fix $r_0\in(0,\pi/2)$ and set
$$
\Sigma^n=\partial B_{\mathbb S^{n+1}}(p,r_0)\times\{0\}.
$$
If $\nu$ is the outward radial normal in $\mathbb S^{n+1}$, then $\sigma=-\cot(r_0)\nu$. Therefore
$$
\int_\Sigma|\sigma|^n\,d\operatorname{vol}_\Sigma
=\cos^n(r_0)|\mathbb S^n|>0
=\theta|\mathbb S^n|,
$$
so the standard Fenchel--Willmore inequality is strict.

For $z\in S_x^-$, write $z=\lambda\nu+\mu\partial_t$, where $\lambda\ge0$ and $\lambda^2+\mu^2=1$. Then $\langle-\sigma,z\rangle=|\sigma|\lambda$. Along the normal geodesic $\gamma_{x,z}(s)=\exp_x(sz)$, we have $\overline{\operatorname{Ric}}(\gamma_{x,z}',\gamma_{x,z}')=n\lambda^2$, and hence
$$
\mathcal B_{\overline{\operatorname{Ric}},R}(x,z)
=\frac{n\lambda^2R^2}{6}.
$$
Thus $\mathcal B_{\overline{\operatorname{Ric}}}(x,z)=+\infty$ whenever $\lambda>0$, so $1-e^{-\mathcal B_{\overline{\operatorname{Ric}}}(x,z)}=1$ for almost every $z\in S_x^-$. By Lemma~\ref{lem:spherical-beta-integral}, the first curvature remainder is
$$
\frac{1}{C_{n,2}}\int_\Sigma\int_{S_x^-}
\langle-\sigma,z\rangle^n
\left(1-e^{-\mathcal B_{\overline{\operatorname{Ric}}}(x,z)}\right)
\,dS_z\,d\operatorname{vol}_\Sigma
=\int_\Sigma|\sigma|^n\,d\operatorname{vol}_\Sigma.
$$

Moreover, for $\lambda>0$,
$$
0\le e^{-\mathcal B_{\overline{\operatorname{Ric}},R}(x,z)}
\min\{\mathcal H_{n,R}(x,z),|\sigma|\lambda\}
\le|\sigma|\lambda e^{-n\lambda^2R^2/6}\longrightarrow0,
$$
while for $\lambda=0$ the minimum vanishes. Hence $\mathcal T_n(x,z)=0$ and the second curvature remainder is zero.

Since $\theta=0$, both sides of \eqref{eq:quantitative-Willmore} equal $\int_\Sigma|\sigma|^n\,d\operatorname{vol}_\Sigma$. Thus equality holds in the quantitative inequality although the standard Fenchel--Willmore inequality is strict.
\end{example}

\section{Sobolev and isoperimetric inequalities under nonnegative intermediate Ricci curvature}
\label{sec:isoperimetric}
Brendle~\cite[Theorem~1.4]{Brendle2023} established the sharp Michael--Simon inequality for submanifolds in ambient manifolds with nonnegative sectional curvature. Ma and Wu~\cite[Theorem~1.1]{MaWu2024} subsequently proved the same inequality under the weaker condition $\overline{\operatorname{Ric}}_k\ge0$, where $k=\min\{n-1,m-1\}$. The result proved in this section removes the dependence of the intermediate Ricci curvature condition on the codimension, requiring only $\overline{\operatorname{Ric}}_{n-1}\ge0$.

\subsection{Sobolev inequality}
Let $M^{n+m}$ be a complete noncompact Riemannian manifold, and let
$\Sigma^n\subset M^{n+m}$ be compact, possibly with boundary, where
$n\ge2$ and $m\ge2$.

Let $f$ be a positive smooth function on
$\Sigma$.
We first assume that $\Sigma$ is connected.
After multiplying $f$ by a positive constant, we may assume that
\begin{equation}
\label{eq:Sobolev-normalization}
\int_\Sigma \sqrt{|\nabla^\Sigma f|^2+f^2|H|^2}\,d\operatorname{vol}_\Sigma
+
\int_{\partial\Sigma} f\,d\operatorname{vol}_{\partial\Sigma}
=
 n\int_\Sigma f^{\frac n{n-1}}\,d\operatorname{vol}_\Sigma.
\end{equation}
The compatibility condition gives a solution $u$ of
\begin{equation*}
\begin{cases}
\operatorname{div}_\Sigma(f\nabla^\Sigma u)
=nf^{\frac n{n-1}}-
\sqrt{|\nabla^\Sigma f|^2+f^2|H|^2},&\text{in }\Sigma,\\
\langle\nabla^\Sigma u,\nu\rangle=1,&\text{on }\partial\Sigma,
\end{cases}
\end{equation*}
where the boundary condition is omitted when
$\partial\Sigma=\varnothing$.  Set
\[
\Omega:=\{x\in\Sigma\setminus\partial\Sigma:|\nabla^\Sigma u(x)|<1\}
\]
and
\[
\mathcal U
:=
\left\{(x,y)\in T^\perp\Sigma:
 x\in\Omega,\ |\nabla^\Sigma u(x)|^2+|y|^2<1
\right\}.
\]
We use the gradient--normal exponential map $\Phi_t^u$ and the contact set
$A_r^u$ introduced in Section~\ref{sec:HK}.
Recall the following standard consequences of the ABP construction; See \cite[Lemmas 4.1, 4.2]{Brendle2023}.

\begin{lemma}
\label{lem:ABP-basic-Sobolev}
\begin{enumerate}
\item If $x\in\Omega$ and $y\in T_x^\perp\Sigma$ satisfy
$|\nabla^\Sigma u(x)|^2+|y|^2\le1$, then
\begin{equation*}
\Delta_\Sigma u(x)-\langle H(x),y\rangle
\le
n f(x)^{\frac1{n-1}}.
\end{equation*}
\item For $0<\alpha<1$ and $r>0$, define
\[
\mathcal C_{\alpha,r}
:=
\left\{p\in M:
\alpha r<d(p,x)<r
\text{ for every }x\in\Sigma
\right\}.
\]
Then
\[
\mathcal C_{\alpha,r}
\subset
\Phi_r^u
\left(
A_r^u\cap\mathcal U\cap
\left\{(x,y):|\nabla^\Sigma u(x)|^2+|y|^2>\alpha^2\right\}
\right).
\]
\end{enumerate}
\end{lemma}

\begin{lemma}
\label{lem:ABP-common-annulus-asymptotics}
    Suppose that $M$ has asymptotic volume ratio $\theta$. Then,
for every fixed $0<\alpha<1$,
\[
\lim_{r\to\infty}
\frac{\operatorname{vol}(\mathcal C_{\alpha,r})}{r^{n+m}}
=
|\mathbb B^{n+m}|(1-\alpha^{n+m})\theta.
\]
\end{lemma}

\begin{proof}
Fix $o\in M$ and put
\[
R_0:=\max_{x\in\Sigma}d(o,x).
\]
For all sufficiently large $r$,
\[
B_o(r-R_0)\setminus {\overline{B_o(\alpha r+R_0)}}
\subset
\mathcal C_{\alpha,r}
\subset
B_o(r+R_0)\setminus B_o(\alpha r-R_0).
\]
    Dividing by $r^{n+m}$ and using the definition of $\theta$ proves the claim.
\end{proof}

\begin{lemma}
\label{lem:mixed-direction-nonnegative}
Assume that $\overline{\operatorname{Ric}}_{n-1}\ge0$.  Let
\[
v:=\nabla^\Sigma u(x)+y,
\qquad
\gamma(s):=\exp_x(sv),
\]
and let $E_1(s),\ldots,E_n(s)$ be the parallel transports of an
orthonormal basis of $T_x\Sigma$.  Then
\[
\overline{\operatorname{Ric}}(\gamma'(s),\gamma'(s))\ge0
\qquad\text{and}\qquad
\mathcal R_n^u(x,y,s)\ge0.
\]
\end{lemma}

\begin{proof}
    By the monotonicity of intermediate Ricci curvature, $\overline{\operatorname{Ric}}_{n-1} \geq 0$ implies $\overline{\operatorname{Ric}}_{n} \geq 0$ and $\overline{\operatorname{Ric}}(\gamma'(s),\gamma'(s))\geq0$.
We now show that $\mathcal R_n^u(x,y,s)\geq0.$

Put $a:=|v|$.  If $a=0$, both quantities vanish.  Assume $a>0$ and set
$w:=\gamma'/a$.  Choose the tangent basis so that the tangential projection
of $w(0)$ is parallel to $E_1(0)$.  Choose a parallel unit vector
$E_{n+1}$ in the direction of the normal projection when it is nonzero.
There are constants $\cos\vartheta,\sin\vartheta\ge0$ such that
\[
w=\cos\vartheta\,E_1+\sin\vartheta\,E_{n+1}.
\]
Set
\[
\xi=-\sin\vartheta\,E_1+\cos\vartheta\,E_{n+1}.
\]
Then $\xi,E_2,\ldots,E_n$ are orthonormal and perpendicular to $w$.
Writing
\[
K(w,e):=\langle\overline R(w,e)e,w\rangle,
\]
we obtain
\begin{align*}
\frac1{a^2}\mathcal R_n^u(x,y,s)
={}&
\sin^2\vartheta
\left(
K(w,\xi)+\sum_{i=2}^nK(w,E_i)
\right)
+
\cos^2\vartheta
\sum_{i=2}^nK(w,E_i).
\end{align*}
The first parenthesis is an $n$-Ricci trace and the second sum is an
$(n-1)$-Ricci trace.  Both are nonnegative.
\end{proof}

\begin{theorem}[Sobolev inequality under nonnegative $(n-1)$-Ricci curvature]
\label{thm:Sobolev}
Let $M^{n+m}$ be complete and noncompact with
$\overline{\operatorname{Ric}}_{n-1}\ge0$, and let $\theta$ be its
asymptotic volume ratio.  Let $\Sigma^n\subset M^{n+m}$ be compact,
possibly with boundary, and let $f$ be a positive smooth function on
$\Sigma$.  If $n\ge2$ and $m\ge2$, then
\begin{align*}
&
\int_\Sigma\sqrt{|\nabla ^{\Sigma}f|^2+f^2|H|^2}\,d\operatorname{vol}_\Sigma + \int_{\partial\Sigma}f\,d\operatorname{vol}_{\partial\Sigma}
\\
&\qquad\ge
n\left( \frac{(n+m)|\mathbb B^{n+m}|}{m|\mathbb B^m|} \right)^{\frac{1}{n}} \theta^{\frac{1}{n}} \left( \int_\Sigma f^{\frac n{n-1}}\,d\operatorname{vol}_\Sigma \right)^{\frac{n-1}{n}}.
\end{align*}
\end{theorem}

\begin{proof}
Assume first that $\Sigma$ is connected and that
\eqref{eq:Sobolev-normalization} holds.  Fix $r>0$ and
$(x,y)\in A_r^u\cap\mathcal U$.
    By Lemma~\ref{lem:mixed-direction-nonnegative},
$$
\overline{\operatorname{Ric}}\ge0
\qquad\text{and}\qquad
\mathcal R_n^u(x,y,t)\ge0
\quad\text{for }0<t<r.
$$
Thus Proposition~\ref{prop:direct-riccati-normal-jacobian}, together
with $R_t(x,y)\ge0$, gives the corresponding Jacobian estimate for
$0<t<r$.
Letting $t\uparrow r$ and using continuity, we obtain
\begin{align*}
|\det D\Phi_r^u(x,y)|
&\le
r^m
\left[
1+
\frac r n
\bigl(\Delta_\Sigma u(x)-\langle H(x),y\rangle\bigr)
\right]_+^n
\\
&\le
r^m
\left(
1+r f(x)^{\frac1{n-1}}
\right)^n.
\end{align*}

By Lemma~\ref{lem:ABP-basic-Sobolev} and the area formula,
\begin{align*}
\operatorname{vol}(\mathcal C_{\alpha,r})
\le{}&
\int_\Omega
\int_{\{y:\alpha^2<|\nabla^\Sigma u(x)|^2+|y|^2<1\}}
 r^m
\left(1+r f(x)^{\frac1{n-1}}\right)^n
\,dy\,d\operatorname{vol}_\Sigma(x).
\end{align*}
For each $x\in\Omega$, the normal-fiber region has volume
\[
|\mathbb B^m|
\left[
(1-|\nabla^\Sigma u(x)|^2)_+^{\frac{m}{2}}
-
(\alpha^2-|\nabla^\Sigma u(x)|^2)_+^{\frac{m}{2}}
\right]
\le
\frac m2|\mathbb B^m|(1-\alpha^2),
\]
where we used
$$
b^{\frac{m}{2}}-a^{\frac{m}{2}} \leq \frac{m}{2}(b-a),
\qquad 0\leq a\leq b\leq 1.
$$
Consequently,
\[
\operatorname{vol}(\mathcal C_{\alpha,r})
\le
\frac m2|\mathbb B^m|(1-\alpha^2)
\int_\Sigma
r^m\left(1+r f^{\frac1{n-1}}\right)^n
\,d\operatorname{vol}_\Sigma.
\]
Divide by $r^{n+m}$ and let $r\to\infty$.  By
Lemma~\ref{lem:ABP-common-annulus-asymptotics} and dominated convergence,
\[
|\mathbb B^{n+m}|(1-\alpha^{n+m})\theta
\le
\frac m2|\mathbb B^m|(1-\alpha^2)
\int_\Sigma f^{\frac n{n-1}}\,d\operatorname{vol}_\Sigma.
\]
Dividing by $1-\alpha$ and letting $\alpha\uparrow1$ gives
\begin{equation*}
(n+m)|\mathbb B^{n+m}|\theta
\le
m|\mathbb B^m|
\int_\Sigma f^{\frac n{n-1}}\,d\operatorname{vol}_\Sigma.
\end{equation*}
Together with \eqref{eq:Sobolev-normalization}, this gives the desired inequality when $\Sigma$ is connected.

If $\Sigma$ is disconnected, apply the connected case to each
component and sum. Since
$$
\sum_j a_j^{\frac{n-1}{n}} \geq \left(\sum_j a_j\right)^{\frac{n-1}{n}} \qquad (a_j\ge0),
$$
the result follows.
\end{proof}

    \begin{theorem}
\label{thm:Sobolev-equality}
Assume $\theta>0$ and equality holds in Theorem~\ref{thm:Sobolev} in the case $m=2$.
Then $f$ is constant and $M$ is isometric to $\mathbb R^{n+2}$.
Moreover, under this identification, $\Sigma$ is an $n$-dimensional round ball contained in an affine $n$-plane.
\end{theorem}

To prove Theorem~\ref{thm:Sobolev-equality}, from now on until the end of this Subsection, we are going to assume that $\theta>0$, $m=2$, and the equality in Theorem \ref{thm:Sobolev} holds.

We follow the proof of \cite[Theorem~1.6]{Brendle2023}, indicating only the modifications needed under the weaker curvature assumption.

As in \cite{Brendle2023}, equality implies that $\Sigma$ is connected. After multiplying $f$ by a positive constant, we may assume that
$$
\int_\Sigma f^{\frac{n}{n-1}}\,d\operatorname{vol}_\Sigma =
|\mathbb B^n|\theta.
$$

We next identify the equality case in the tangential Riccati estimate.

\begin{lemma}
\label{lem:equality-jacobian-lower-bound}
Assume that $x\in\Omega$ and $y\in T_x^\perp\Sigma$ satisfy $|\nabla^\Sigma u(x)|^2+|y|^2=1.
$
Then, for every $t>0$,
$$
|\det D\Phi_t^u(x,y)|
\geq
t^2\left(1+t f(x)^{\frac1{n-1}}\right)^n.
$$
\end{lemma}

\begin{proof}
It suffices to verify the monotonicity used in the proof of \cite[Lemma~5.1]{Brendle2023}.

Fix $r>0$ and $(x,y)\in A_r^u\cap\mathcal U$, and set
$$
a:=f(x)^{\frac{1}{n-1}},
\qquad
c:=\frac{1}{n}\bigl(\Delta_\Sigma u(x)-\langle H(x),y\rangle\bigr).
$$
By Lemma~\ref{lem:mixed-direction-nonnegative} and
Proposition~\ref{prop:direct-riccati-normal-jacobian},
$$
t\longmapsto
\frac{|\det D\Phi_t^u(x,y)|}
{t^2(1+ct)^n e^{-R_t(x,y)}}
$$
is nonincreasing on $(0,r)$.
Moreover,
Lemma~\ref{lem:ABP-basic-Sobolev} gives $c\le a$, and so
$$
t\longmapsto\frac{1+ct}{1+at}
$$
is nonincreasing.
By
Proposition~\ref{prop:direct-riccati-normal-jacobian},
$R_t(x,y)$ is nondecreasing, so
$$
\frac{|\det D\Phi_t^u(x,y)|}
{t^2(1+at)^n}
=
\frac{|\det D\Phi_t^u(x,y)|}
{t^2(1+ct)^n e^{-R_t(x,y)}}
\left(\frac{1+ct}{1+at}\right)^n
e^{-R_t(x,y)}
$$
is nonincreasing on $(0,r)$.

The remainder of the proof is identical to that of
\cite[Lemma~5.1]{Brendle2023}.
\end{proof}

\begin{lemma}
\label{lem:equality-hessian}
Assume that $x\in\Omega$ and $y\in T_x^\perp\Sigma$ satisfy $|\nabla^\Sigma u(x)|^2+|y|^2=1.$
Then
$$
\nabla^2_\Sigma u(x)-\langle\mathrm{II}(x),y\rangle
=
f(x)^{\frac1{n-1}}g.
$$
\end{lemma}

\begin{proof}
We follow the proof of \cite[Lemma~5.2]{Brendle2023}, replacing the
matrix Riccati comparison by Proposition~\ref{prop:direct-riccati-normal-jacobian}.

Set
$$
a:=f(x)^{\frac{1}{n-1}},
\qquad
c:=\frac{1}{n}
\bigl(\Delta_\Sigma u(x)-\langle H(x),y\rangle\bigr).
$$
By Lemma~\ref{lem:ABP-basic-Sobolev}, $c\le a$.
Let $P(t)$ be the Jacobi matrix from the proof of
Proposition~\ref{prop:direct-riccati-normal-jacobian} along
$$
\gamma(t):=\exp_x\bigl(t(\nabla^\Sigma u(x)+y)\bigr).
$$
By Lemma~\ref{lem:equality-jacobian-lower-bound},
$$|\det P(t)|\geq t^2(1+at)^n >0
\qquad\text{for all }t>0.$$
Thus $P(t)$ is invertible for every $t>0$. We can then argue as in Proposition \ref{prop:direct-riccati-normal-jacobian} to show that
$$
|\det P(t)| \leq t^2(1+ct)^n e^{-R_t(x,y)} \leq t^2(1+at)^n.
$$
Therefore equality holds throughout, so
$$
c=a,\qquad R_t(x,y)=0, \qquad |\det P(t)|=t^2(1+at)^n
$$
for every $t>0$.

From \eqref{eq:Rs2}, we have
\begin{equation*}
R_t = \int_0^t W_m(t, \rho)\left[n\delta_n(\rho)+m\delta_m(\rho) \right]\, d\rho + \int_0^t \left[W_n(t, \rho)-W_m(t, \rho) \right] n\delta_n(\rho)\, d\rho = 0.
\end{equation*}
Both weights are strictly positive for $0<\rho<t$,
so the nonnegative continuous functions $n\delta_n$ and
$n\delta_n+m\delta_m$ vanish identically. Hence
$$
\delta_n(t)=\delta_m(t)=0
\qquad\text{for all }t>0.
$$
Consequently, the scalar Riccati comparisons with $a=c$ give
$$
q_n(t)\leq\frac{a}{1+at},
\qquad
q_m(t)\leq\frac{1}{t}.
$$

On the other hand, $|\det P(t)|=t^2(1+at)^n$ implies
$$
nq_n(t)+2q_m(t) = \frac{d}{dt}\log|\det P(t)| = \frac{na}{1+at}+\frac{2}{t}.
$$
Hence
$$
q_n(t)=\frac{a}{1+at},
\qquad
q_m(t)=\frac{1}{t}.
$$

Write the symmetric matrix $Q$ in block form as
$$ Q=
\begin{pmatrix}
Q_T&B\\
B^T&Q_N
\end{pmatrix}.
$$
Taking the tangential trace of the Riccati equation gives
$$
q_n' + \frac{1}{n}\operatorname{tr}(Q_T^2) + \frac{1}{n}|B|^2 + \delta_n=0.
$$
Since $q_n'+q_n^2=0$ and $\delta_n=0$, we obtain
$$
\operatorname{tr}(Q_T^2)+|B|^2=nq_n^2.
$$
Then the equality in Cauchy--Schwarz inequality
$\operatorname{tr}(Q_T^2)\ge nq_n^2$ holds, which implies
$$
Q_T(t)=q_n(t)I_n.
$$
Letting $t\to0^+$ gives
$$
\nabla^2_\Sigma u(x)-\langle\mathrm{II}(x),y\rangle = ag = f(x)^{\frac{1}{n-1}}g.
$$
\end{proof}

With Lemma~\ref{lem:equality-hessian} established, the remainder of the argument is the same as in the proof of
\cite[Theorem~1.6]{Brendle2023}.
Indeed, applying Lemma~\ref{lem:equality-hessian} to $y$ and $-y$ gives
$$
\nabla^2_\Sigma u=f^{\frac1{n-1}}g,
\qquad
\mathrm{II}=0
$$
on $\Omega$.
The argument of \cite[Lemmas~5.4 and~5.5]{Brendle2023}
then shows that $\nabla^\Sigma f=0$ on $\Omega$ and that $\Omega$ is dense in $\Sigma$. Hence $f$ is constant,
$$
\nabla^2_\Sigma u=f^{\frac{1}{n-1}}g,
\qquad
\mathrm{II}=0
$$
on $\Sigma$.
The remaining flow argument shows that $\theta=1$.
Since $\overline{\operatorname{Ric}}_{n-1}\geq 0$ implies
$\overline{\operatorname{Ric}}\geq 0$, the rigidity case of the
Bishop--Gromov theorem yields that $M$ is isometric to
$\mathbb R^{n+2}$. Finally, $\Sigma$ is an $n$-dimensional round ball contained in an affine $n$-plane. This completes the proof of Theorem~\ref{thm:Sobolev-equality}.

\subsection{Applications of the Sobolev inequality}

Taking $f\equiv1$ in Theorem~\ref{thm:Sobolev} gives the corresponding isoperimetric inequality for
minimal submanifolds.

\begin{corollary}[Isoperimetric inequality for minimal submanifolds]
Under the assumptions of Theorem~\ref{thm:Sobolev}, if $\Sigma$ is minimal,
then
\[
|\partial\Sigma|
\ge
n\left(
\frac{(n+m)|\mathbb B^{n+m}|}{m|\mathbb B^m|}
\right)^{\frac{1}{n}}
\theta^{\frac{1}{n}}|\Sigma|^{\frac{n-1}{n}}.
\]
For $m=2$, this becomes
\[
|\partial\Sigma|
\ge
n|\mathbb B^n|^{\frac{1}{n}}\theta^{\frac{1}{n}}
|\Sigma|^{\frac{n-1}{n}}.
\]
    Moreover, if $\theta>0$, $m=2$, and equality holds, then $M$ is isometric to $\mathbb R^{n+2}$ and, under this identification, $\Sigma$ is an $n$-dimensional round ball contained in an affine $n$-plane.
\end{corollary}

We next introduce an intrinsic diameter estimate. Wu~\cite{wu2023diameter} obtained such an estimate from a Michael--Simon Sobolev inequality under nonnegative sectional curvature by combining a maximal-function argument with a Vitali-type covering lemma. Once the required Sobolev inequality is available, the remaining argument is intrinsic to the submanifold and does not use the sectional-curvature assumption on the ambient manifold. Thus, Theorem~\ref{thm:Sobolev} yields the following corollary under nonnegative ($n-1$)-Ricci curvature.

\begin{corollary}[Intrinsic diameter estimate]\label{cor:intrinsic-diameter}
Let $M^{n+m}$ be complete and noncompact with nonnegative $(n-1)$-Ricci curvature and asymptotic volume ratio $\theta >0$. Let $\Sigma^n$ be a closed connected submanifold of $M$, where $n\geq2$ and $m\geq2$.  Then there is a constant
$$
        C(n,m,\theta)
        :=
        \min\left\{
        \frac{(n+m)|\mathbb B^{n+m}|\theta}{2^n m|\mathbb B^m|},
        |\mathbb B^n|,
        1
        \right\}
$$
such that
\[\operatorname{diam}_{\Sigma}(\Sigma)
        \le
        4C(n,m,\theta)^{1-n}
        \int_{\Sigma}|H|^{n-1}\,d \mathrm{vol}_{\Sigma}.
\]
\end{corollary}

The local argument for Corollary~\ref{cor:intrinsic-diameter} also gives
the following compactness criterion.

\begin{corollary}
    Let $M^{n+m}$ be complete, noncompact, simply connected with nonnegative $(n-1)$-Ricci curvature and asymptotic volume ratio $\theta >0$. Let $\Sigma^n$ be a complete connected submanifold in $M$.
    If there is a point $p \in \Sigma$ such that $$\mathrm{vol}_\Sigma(B_p^\Sigma(R)) = o(R^n) \text{ as } R \to \infty, \quad \int_\Sigma |H|^{n-1} d\operatorname{vol}_\Sigma < +\infty,$$
    then $\Sigma$ is compact.
\end{corollary}
We refer the reader to the proof of \cite[Corollary~3.2]{wu2023diameter}, since the same argument applies here. We therefore omit the proof.

Theorem~\ref{thm:Sobolev} also yields a boundary version of the intrinsic diameter estimate.
Wu~\cite[Theorem~1.4]{wu2023diameter} proved such an estimate for
compact \textit{convex}\footnote{To our understanding, the term “convex” in \cite{wu2023diameter} refers to geodesic convexity.} surfaces with boundary.
The same argument can be adapted without assuming that $\Sigma$ is convex.

For $m\geq 1$, let
\begin{equation*}
c(m):=\frac{1}{2}\left(\frac{m|\mathbb B^m|}{(m+2)|\mathbb B^{m+2}|}\right)^{\frac12},
\end{equation*}
the Sobolev constant appearing in Theorem~\ref{thm:Sobolev} for $n=2$. For $\theta>0$, define
\begin{equation}\label{delta-definition}
\delta(m,\theta):=
\begin{cases}
\frac{\theta}{16c(m)^2}, & \text{if } m\geq 2,\\
\frac{\theta}{16c(2)^2}, & \text{if }m=1.
\end{cases}
\end{equation}
\begin{corollary}\label{coro:surface-diameter}
    Let $M^{m+2}$ be a complete noncompact Riemannian manifold with nonnegative sectional curvature and asymptotic volume ratio $\theta>0$.
    Let $\Sigma^2 \subset M$ be a compact connected surface, possibly with boundary.
    If $m\geq 1$, then
    $$\operatorname{diam}_{\Sigma}(\Sigma) \leq \frac{4}{\delta(m,\theta)}\left( \int_\Sigma |H| d \mathrm{vol}_{\Sigma} + L(\partial \Sigma)\right),$$
    where $\delta(m,\theta)>0$ is defined in \eqref{delta-definition}.
\end{corollary}

\begin{remark}
When the dimension of the submanifold is two, the curvature
condition in Theorem~\ref{thm:Sobolev} becomes
$\operatorname{Ric}_{1}\geq0$, which is precisely nonnegative
sectional curvature. Thus, this corollary does not weaken the
ambient curvature assumption in Wu's theorem; the point here is
that no convexity assumption on $\Sigma$ is required.
We also mention that Flaim and Scharrer ~\cite[Lemma 2.4]{flaim2024diameter} obtained a related convexity-free diameter estimate for surfaces with boundary in conformally flat ambient spaces, using Miura's doubling construction.
\end{remark}

We begin with a modification of \cite[Lemma~4.1]{wu2023diameter}. In the proof of Wu’s lemma, convexity is used to ensure that the intrinsic distance function $d_\Sigma(x,\cdot)$ and the area function $A\bigl(B^\Sigma(x,r)\bigr)$ are Lipschitz. The first property, however, does not require convexity, since every distance function is $1$-Lipschitz with respect to the intrinsic metric. For the second, we modify the argument so that neither the local Lipschitz continuity nor the almost-everywhere differentiability of the area function is needed. This allows us to remove the convexity assumption. We also extend the conclusion from interior points to all points of $\Sigma$, including points on $\partial\Sigma$.
\begin{lemma}
    Let $M^{2+m}$ be a complete noncompact Riemannian manifold with nonnegative sectional curvature and asymptotic volume ratio $\theta>0$.
    Let $\Sigma^2 \subset M$ be a compact connected surface, possibly with boundary.
    Then for any $x \in \Sigma$ and $R>0$, at least one of the following holds:
    \begin{enumerate}
        \item $$M(x,R) := \sup_{r \in (0,R]} r^{-1}\left(\int_{B^\Sigma(x,r)}|H| d \mathrm{vol}_{\Sigma} + L(B^\Sigma(x,r) \cap \partial \Sigma)\right)\geq \delta(m,\theta);$$
        \item $$\kappa(x,R) := \inf_{r \in (0,R]}\frac{A(B^\Sigma(x,r))}{r^2}\geq \delta(m,\theta).$$
    \end{enumerate}
    Here, $\delta(m,\theta)>0$ is defined in \eqref{delta-definition}, $B^\Sigma(x,r)$ denotes the intrinsic open ball in $\Sigma$ centered at $x$ with radius $r$, and $A$ and $L$ denote the area and length with respect to the induced metric on $\Sigma$, respectively.
\end{lemma}
\begin{proof}
Let $\delta:=\delta(m,\theta)$.
Suppose this is not true, then there exists $x\in \Sigma$ such that for some
$R>0$,
we have both
\begin{align}\label{eq:M<delta}
M(x,R)<\delta
\end{align}
and
\begin{equation}\label{eq:kappa}
\kappa(x,R)<\delta.
\end{equation}
It is easy to see that $x\in \Sigma\setminus\partial \Sigma$, for otherwise the fact that $\lim _{r \rightarrow 0^{+}} \frac{L\left(B^{\Sigma}(x, r) \cap \partial \Sigma\right)}{r}=2>\delta$ would violate \eqref{eq:M<delta}. Note also that $\delta(m,\theta)$ is non-increasing in $m$.

    Then, for every $r \in (0,R]$,
    \begin{equation}\label{ineq:r-eta'}
        \int_{B^\Sigma (x,r)}|H| d \mathrm{vol}_{\Sigma} + L(B^\Sigma (x,r) \cap \partial \Sigma)  <r \delta.
    \end{equation}

    Let $A(r):=A\bigl(B^\Sigma(x,r)\bigr)$
    and define $v(r) := \delta r^2$. Since $\delta < \pi$ and $\lim_{r\to 0^+}\frac{A(r)}{r^2} = \pi$, we have $A(r) > v(r)$ for all sufficiently small $r>0$.

By \eqref{eq:kappa}, there exists $r \in (0,R]$ such that $A(r)\leq v(r)$.
    Define
    $$r_0 : = \inf \{r \in (0, R] \, |\, A(r) \leq v(r)\}.$$
    Then $r_0 >0$ and $A(r) > v(r)$ for all $r \in (0, r_0)$.
    Since $A$ is continuous, $A(r_0) = v(r_0)$.
    Let us assume that $m\ge2$ at the moment.
    For small $s\in(0,r_0)$, define
    $$h_s(y) : = \begin{cases}
        1 &\mbox{ if }d_\Sigma (x,y) <r_0-s\\
        \frac{r_0 - d_\Sigma (x,y)}{s} &\mbox{ if }r_0-s \leq d_\Sigma (x,y) < r_0\\
        0 &\mbox{ if } d_\Sigma (x,y) \geq r_0.
    \end{cases}$$
    Since $d_\Sigma(x,\cdot)$ is 1-Lipschitz, $h_s$ is Lipschitz and $|\nabla^{\Sigma} h_s| \leq \frac{1}{s}$ almost everywhere.

     Applying Theorem \ref{thm:Sobolev} to \(h_s\) and using a standard smooth approximation if necessary, we obtain
    $$A(r_0-s)^\frac{1}{2} \leq c(m) \theta^{-\frac{1}{2}} \left( \frac{A(r_0) - A(r_0-s)}{s} + \int_{B^\Sigma (x,r_0)}|H| d \mathrm{vol}_{\Sigma}  + L(B^\Sigma (x,r_0) \cap \partial \Sigma) \right).$$
    Combined with \eqref{ineq:r-eta'}, we have
    \begin{equation}\label{eq:A difference}
    \begin{split}
        A(r_0) - A(r_0-s) &> s \left( A(r_0 -s)^\frac{1}{2} c(m)^{-1}\theta^\frac{1}{2} - r_0 \delta\right)\\
        & > s \left({{3}} \delta r_0-4\delta s \right),
    \end{split}
    \end{equation}
    where we used $A\left(r_0-s\right)>v\left(r_0-s\right)=\delta\left(r_0-s\right)^2$ and $c(m)^{-1} \theta^{\frac{1}{2}} {\delta}^{\frac{1}{2}}={{4}}  \delta$.

    Since $v\left(r_0-s\right)=v\left(r_0\right)-s\left(2 r_0 \delta-s \delta\right)$
    and $A(r_0) = v(r_0)$, it follows that from \eqref{eq:A difference} that
\begin{equation*}
\begin{split}
0>v\left(r_0-s\right)-A\left(r_0-s\right)
>&s\left(r_0 \delta-3 s \delta\right).
\end{split}
\end{equation*}
    Choosing
    $ 0 <s < {{\frac{r_0 }{3}}}$
    makes the last expression positive, which is a contradiction.

    Suppose now $m=1$. By taking the product of $M$ with $\mathbb R$, we can assume that $\Sigma$ is a codimension-two submanifold, which still has non-negative sectional curvature and the same asymptotic volume ratio $\theta>0$. The previous argument can then be carried out to show that either $M(x,R)\ge \delta(2,\theta)$ or $\kappa(x,R)\ge \delta(2,\theta)$ for any $x\in \Sigma$ and $R>0$.
\end{proof}

\begin{proof}[Proof of Corollary~\ref{coro:surface-diameter}]
    We follow the proof of \cite[Theorem~1.4]{wu2023diameter}. In that argument, the convexity of $\Sigma$ is used only to ensure that a curve realizing the intrinsic diameter is contained in $\Sigma\setminus\partial\Sigma$, where \cite[Lemma~4.1]{wu2023diameter} can be applied. Since the preceding lemma holds for every $p\in\Sigma$, including points of $\partial\Sigma$, the construction of the radius $s(p)$ remains valid along an intrinsic distance-minimizing curve even when the curve meets the boundary. Moreover, the Vitali-type covering argument used in the proof depends only on the metric properties of the minimizing curve and does not require it to be a smooth geodesic contained in the interior. The remainder of the proof is therefore identical to that of \cite[Theorem~1.4]{wu2023diameter}.
\end{proof}

\begin{remark}
{The convexity assumption can likewise be removed from the diameter
argument in \cite[Theorem~2]{paeng2014}, provided the hypotheses for the
global Hoffman--Spruck inequality are retained. This applies, in particular,
in a complete simply connected ambient manifold with nonpositive sectional
curvature; see \cite[Theorem~2.1 and Remark~2.2]{flaim2024diameter}.}
\end{remark}

\section{Curvature-decay extensions}\label{sec:decay}

    In this section, we extend the Willmore and Michael--Simon Sobolev inequalities obtained in Sections~\ref{sec:quantitative-Willmore} and~\ref{sec:isoperimetric} to the quadratic curvature-decay setting introduced in Section~\ref{sec:introduction}. We begin with some preliminaries for the $(\mathrm{QD})_k$ setting.
\begin{lemma}
\label{lem:QD-monotonicity}
Let $1\le k\le \ell\le N-1$.
If $(M^N,\overline g)$ satisfies $(\mathrm{QD})_k$ with respect to
$o\in M$ and $\lambda$, then it also satisfies $(\mathrm{QD})_\ell$
with respect to the same $o$ and $\lambda$.
\end{lemma}
This follows by summing the $k$-trace bounds over all $k$-element
subsets of an orthonormal $\ell$-frame perpendicular to a unit vector:
each sectional term occurs $\binom{\ell-1}{k-1}$ times.

Let $h_1$ be the solution of
\begin{equation*}
h_1''(t)=\lambda (t) h_1 (t),
\qquad
h_1(0)=0,
\qquad
h_1'(0)=1.
\end{equation*}
By \cite[Lemma~2.7]{ChongLuoLu2025}, for every $c\in \mathbb R$, we have
\begin{equation}\label{property:h_1-limit}
    \lim_{t \to \infty} \frac{h_1(t+c)}{h_1(t)} = 1.
\end{equation}
We define the generalized asymptotic volume ratio associated with $\lambda$ by
\begin{equation}
\label{eq:model-AVR-decay}
\theta_\lambda
:=
\lim_{r\to\infty}
\frac{\operatorname{vol}(B_o(r))}
{N|\mathbb B^N|\int_0^r h_1(t)^{N-1}\,dt}.
\end{equation}
Since Lemma~\ref{lem:QD-monotonicity} with
$\ell=N-1$ gives
$$
    \overline{\operatorname{Ric}}
    \ge -(N-1)\lambda(d(o,\cdot)) \,\overline g.
$$
the radial Bishop--Gromov comparison \cite[Theorem 2.14]{pigola2008vanishing} implies that the ratio in
\eqref{eq:model-AVR-decay} is nonincreasing. In particular, $\theta_\lambda$ is well defined.
Moreover, by \eqref{property:h_1-limit},  the center $o$ may be replaced by any $x \in M$.

We now use the factorization \eqref{eq:gradient-normal-det-factorization} to estimate the gradient--normal exponential map.

Fix $(x,y) \in A_r^u$ and set $v:= \nabla^\Sigma u(x) + y$ and
$$\gamma (t) : = \exp_x (tv) \quad \text{ for } 0 \leq t \leq r.$$
To accommodate both the Willmore and Sobolev applications, we do not impose a specific intermediate Ricci curvature condition at this stage.
Instead, we assume along $\gamma$ that
\begin{equation}
\label{eq:QD-curvature-trace-assumptions}
\begin{aligned}
\mathcal R_n^u(x,y,t)
&\ge
-n\lambda(d(o,\gamma(t)))|v|^2,\\
\overline{\operatorname{Ric}}(\gamma'(t),\gamma'(t))
&\ge
-(N-1)\lambda(d(o,\gamma(t)))|v|^2,
\end{aligned}
\qquad 0\le t\le r.
\end{equation}

We first estimate the tangential factor $\mathrm{det} B_t^u$ in \eqref{eq:gradient-normal-det-factorization}.
Recall that $b_1$ is defined in \eqref{eq:b1}.
Set $\lambda_x(s):=\lambda(|d(o,x)-s|)$ and define
\[
\mu_x:=\int_0^\infty\lambda_x(s)\,ds
=b_1+\int_0^{d(o,x)}\lambda(s)\,ds.
\]

\begin{lemma}\label{lem:quadratic-B-estimate}
    For every $0<t<r$,
\begin{equation}
\label{eq:Bt-shifted-bound}
\det B_t^u
\le
\left[
1+t c_u(x,y)+t\mu_x|v|
\right]_+^n,
\end{equation}
where $B_t^u$ is defined in \eqref{eq:block-gradient-normal-decomposition} and $c_u(x, y):=\frac{1}{n}\left(\Delta_{\Sigma} u(x)-\langle H(x), y\rangle\right)$.
\end{lemma}

\begin{proof}
    By Lemma~\ref{lem:general-contact-propagation}, $\gamma$ is minimizing
on $[0,r]$, and hence $d(x,\gamma(s))=s|v|$.
The triangle inequality and monotonicity of $\lambda$ imply
\[
\lambda(d(o,\gamma(s)))
\leq
\lambda\big(\big|d(o,x)-s|v|\big|\big)
=
\lambda_x(s|v|).
\]

Using the first inequality in \eqref{eq:QD-curvature-trace-assumptions} to estimate the curvature term of \eqref{eq:det Bt}, together with the above estimate, we obtain
\begin{align*}
-
\frac{t}{n}
\int_0^t
\left(1-\frac{s}{t}\right)^2
\mathcal R_n^u(x,y,s)\,ds
&\le
t|v|^2
\int_0^t
\left(1-\frac{s}{t}\right)^2
\lambda_x(s|v|)\,ds
\\
&\le
t|v|^2
\int_0^t
\lambda_x(s|v|)\,ds
\\
&=
t|v|
\int_0^{t|v|}
\lambda_x(\rho)\,d\rho
\\
&\le
t\mu_x|v|.
\end{align*}
Substituting this into \eqref{eq:det Bt} gives
\eqref{eq:Bt-shifted-bound}.
\end{proof}

Let $h_x$ be the solution of
\begin{equation*}
h_x''(s)=\lambda_x(s) h_x(s),
\qquad
h_x(0)=0,
\qquad
h_x'(0)=1.
\end{equation*}
Since $\lambda_x\ge0$, the function $h_x$ is positive, convex, and increasing on $(0,\infty)$.

\begin{lemma}
\label{lem:decay-ambient-exp}
For any
$0<t<r$,
\begin{equation*}
\left|\det D(\exp_x)_{tv}\right|
\le
\left(\frac{h_x(t|v|)}{t|v|}\right)^{N-1},
\end{equation*}
where the right-hand side is interpreted continuously at $v=0$, with value $1$.
\end{lemma}

\begin{proof}
If $v=0$, then $D(\exp_x)_0=I$, while
\[
\lim_{s\to 0^+}\frac{h_x(s)}{s}=1,
\]
so the assertion is immediate. We therefore assume that $|v|>0$.

Since $(x,y)\in A_r^u$, Lemma~\ref{lem:general-contact-propagation}
shows that $\gamma|_{[0,r]}$ is minimizing. Hence, for $0<t<r$,
there is no point conjugate to $x$ along $\gamma|_{(0,t]}$.

As in Lemma~\ref{lem:gradient-normal-signed-ambient}, define a Jacobi tensor
$$\mathcal J_v(s):v^\perp\longrightarrow\gamma'(s)^\perp$$
satisfying $\mathcal J_v(0) = 0$ and $\mathcal J_v'(0) = I$.
Then $\mathcal J_v$ is nonsingular and $\det \mathcal{J}_v(s) >0$ for $0 <s\leq t$.
Moreover, as in \eqref{eq:gradient-normal-point-exp-density}, we obtain
\begin{equation*}
        \left|\det D(\exp_x)_{sv}\right|
        =
        \frac{\det\mathcal J_v(s)}{s^{N-1}}.
\end{equation*}

To estimate $\det\mathcal J_v(s)$, we proceed as in the proof of Lemma \ref{lem:gradient-normal-signed-ambient} by replacing the vector field $V_i^{(s)}$ in \eqref{eq:comparison-field} with
\begin{align*}
V_i^{(s)}(\tau):=\frac{h_x(\tau|v|)}{h_x(s|v|)} E_i(\tau).
\end{align*}
The same index-form comparison, using
\eqref{eq:QD-curvature-trace-assumptions} and $h_x''=\lambda_x h_x$, gives
\begin{align*}
\frac{d}{ds}\log\det\mathcal J_v(s)
&\le
\frac{(N-1)|v|^2}{h_x(s|v|)^2}
\int_0^s
\bigl[h_x'(\tau|v|)^2+
\lambda_x(\tau|v|)h_x(\tau|v|)^2\bigr]d\tau\\
&=(N-1)|v|\frac{h_x'(s|v|)}{h_x(s|v|)}.
\end{align*}
Integrating and using
$\det\mathcal J_v(s)/h_x(s|v|)^{N-1}\to |v|^{-(N-1)}$
as $s\downarrow0$ yields
\begin{equation*}
\left| \det  D\left(\exp _x\right)_{t v}\right|=\frac{\det \mathcal{J}_v(t)}{t^{N-1}}    \leq\left(\frac{h_x(t|v|)}{t|v|}\right)^{N-1}.
\end{equation*}
\end{proof}

\begin{proposition}
\label{prop:quadratic-gradient-normal-jacobian}
For every $0<t\leq r$,
\begin{equation}
\label{eq:quadratic-gradient-normal-jacobian}
\left|\det D\Phi_t^u(x,y)\right|
\le
t^m
\left[
1+t c_u(x,y)+t\mu_x|v|
\right]_{{+}}^n
\left(
\frac{h_x(t|v|)}{t|v|}
\right)^{N-1},
\end{equation}
with the usual continuous interpretation when $v=0$.
\end{proposition}

\begin{proof}
{Fix $0<t<r$.} By the factorization
\eqref{eq:gradient-normal-det-factorization},
\[
\left|\det D\Phi_t^u(x,y)\right|
=
t^m\det B_t^u
\left|\det D(\exp_x)_{tv}\right|.
\]
Combining Lemma~\ref{lem:quadratic-B-estimate} with
Lemma~\ref{lem:decay-ambient-exp} gives
\eqref{eq:quadratic-gradient-normal-jacobian}.
Both $D\Phi_t^u(x,y)$ and the right-hand side are continuous in $t$.
Letting $t\uparrow r$ proves the endpoint statement.
\end{proof}

To relate the Jacobian estimate in {Proposition~\ref{prop:quadratic-gradient-normal-jacobian}} to the generalized asymptotic volume ratio, we need to compare the point-dependent model function $h_x$ with the reference model function $h_1$. The following lemma gives such a comparison uniformly for $x \in \Sigma$.

For a compact submanifold $\Sigma$, set
$R_0:=\max_{x\in\Sigma}d(o,x)$ and
\begin{equation*}
L_\Sigma:=2e^{b_1R_0}-1.
\end{equation*}
\begin{lemma}
\label{lem:decay-shifted-model}
For every $\varepsilon>0$, there exists $T_\varepsilon>0$, independent of
$x\in\Sigma$, such that
\begin{equation*}
h_x(t)
\le
(L_\Sigma+\varepsilon)h_1(t)
\end{equation*}
for every $x\in\Sigma$ and every $t\ge T_\varepsilon$.
\end{lemma}

\begin{proof}
The case $b_1=0$ is immediate. Assume $b_1>0$, and let $h_2$ be the
companion solution
\[
h_2''=\lambda h_2,\qquad h_2(0)=1,\qquad h_2'(0)=0.
\]
Since $d(o,x)\le R_0$, \cite[Lemma 2.5]{ChongLuoLu2025} gives
\[
h_x(t)
\le
\frac{e^{R_0b_1}-1}{b_1}h_2(t)
+
e^{R_0b_1}h_1(t).
\]
Moreover, \cite[Lemma 2.6]{DongLinLu2024} implies
\[
\limsup_{t\to\infty}\frac{h_2(t)}{h_1(t)}
\le b_1.
\]
Consequently,
\[
\limsup_{t\to\infty}\sup_{x\in\Sigma}
\frac{h_x(t)}{h_1(t)}
\le
\frac{e^{R_0b_1}-1}{b_1}b_1+e^{R_0b_1}
=
2e^{R_0b_1}-1
=
L_\Sigma,
\]
which proves the claim.
\end{proof}

\begin{theorem}[Willmore-type inequality under quadratic curvature decay]
\label{thm:decay}
    Let $(M^N,\overline g)$ be a complete noncompact Riemannian manifold, where $N=n+m$, satisfying $(\mathrm{QD})_{n}$ with respect to $o\in M$ and $\lambda$.
Let
$\Sigma^n\subset M^N$ be a closed immersed submanifold. Then
\begin{equation}
\label{eq:shifted-Willmore-decay}
\int_\Sigma
\int_{S_x^\perp\Sigma}
\left(
\mu_x-\langle\sigma(x),\omega\rangle
\right)_+^n
\,d\omega\,d\operatorname{vol}_\Sigma(x)
\ge
\frac{N|\mathbb B^N|\theta_\lambda}{L_\Sigma^{N-1}}.
\end{equation}
Equivalently,
\begin{equation}\label{eq:Willmore-decay-error-form}
\int_\Sigma|\sigma|^n\,d\operatorname{vol}_\Sigma \geq
\frac{|\mathbb S^n|\theta_\lambda}{L_\Sigma^{N-1}}
-
\mathcal E_{\lambda,\Sigma},
\end{equation}
where
\begin{equation}\label{eq:decay-error-definition}
\mathcal E_{\lambda,\Sigma}
:=
\frac1{C_{n,m}}
\int_\Sigma
\int_{S_x^\perp\Sigma}
\bigg[
\left(
\mu_x-\langle\sigma(x),\omega\rangle
\right)_+^n
-
\left(
-\langle\sigma(x),\omega\rangle
\right)_+^n
\bigg]
\,d\omega\,d\operatorname{vol}_\Sigma(x).
\end{equation}
Moreover,
\[
\mathcal E_{\lambda,\Sigma}\ge0.
\]
\end{theorem}

    \begin{remark}\label{rem:decay-correction-necessary}
When $\lambda\equiv 0$, the condition $(\mathrm{QD})_n$ reduces to
$\overline{\operatorname{Ric}}_n\ge 0$. In this case,
\[
    \mu_x=0,
    \qquad
    L_\Sigma=1,
\]
and $\theta_\lambda$ coincides with the usual asymptotic volume ratio
$\theta$. Hence $\mathcal E_{\lambda,\Sigma}=0$, and \eqref{eq:Willmore-decay-error-form}
recovers the corresponding Willmore inequality under nonnegative
$n$-intermediate Ricci curvature.

The correction term in \eqref{eq:Willmore-decay-error-form} is necessary in the
presence of curvature decay. Indeed, let
\[
    M^{n+1}=\mathbb R\times \mathbb S^n,
    \qquad
    g=dr^2+(1+r^2)g_{\mathbb S^n},
\]
where $g_{\mathbb S^n}$ is the round metric of sectional curvature $1$, and fix
$o=(0,p_0)$. The warped-product curvature formulas give
\[
    \overline{\operatorname{Ric}}(\partial_r,\partial_r)
    =-\frac{n}{(1+r^2)^2},
    \qquad
    \overline{\operatorname{Ric}}(U,U)
    =\frac{n-2}{(1+r^2)^2}
\]
for every unit vector $U$ tangent to the spherical factor. Hence
\[
    \overline{\operatorname{Ric}}_n
    \ge -\frac{n}{(1+r^2)^2}g.
\]
If $q=(r,p)$ and $\delta=d_{\mathbb S^n}(p_0,p)\le\pi$, a path that moves in both
factors gives
\[
    d(o,q)^2\le r^2+(1+r^2)\delta^2,
\]
and therefore
\[
    1+d(o,q)^2\le(1+\pi^2)(1+r^2).
\]
Thus $M$ satisfies $(\mathrm{QD})_n$ with
\[
    \lambda(t)=\frac{(1+\pi^2)^2}{(1+t^2)^2}.
\]
Moreover, $|r|\le d(o,q)\le |r|+\pi$ and
$d\operatorname{vol}_M=(1+r^2)^{\frac{n}{2}}dr\,d\operatorname{vol}_{\mathbb S^n}$, so
$M$ has Euclidean volume growth. Since
$\int_0^\infty t\lambda(t)\,dt<\infty$, the model function is
asymptotically linear, and hence $\theta_\lambda>0$. On the other hand,
\[
    \Sigma_0=\{0\}\times \mathbb S^n
\]
is totally geodesic, so that $\sigma\equiv0$ on $\Sigma_0$. Consequently,
no positive lower bound of the classical form
\[
    \int_\Sigma |\sigma|^n\,d\operatorname{vol}_\Sigma
    \ge C\theta_\lambda,
    \qquad C>0,
\]
can hold under the curvature-decay assumption alone. Thus the subtractive
correction term in \eqref{eq:Willmore-decay-error-form} reflects a genuine
feature of the curvature-decay setting.
\end{remark}

\begin{proof}[Proof of Theorem~\ref{thm:decay}]

Fix $(x,\omega)\in S^\perp\Sigma$ and
$0<t<\tau_\Sigma(x,\omega)$. Choose $t<r<\tau_\Sigma(x,\omega)$;
then $(x,\omega)\in A_r^0$. The condition $(\mathrm{QD})_n$ and
Lemma~\ref{lem:QD-monotonicity} give the two bounds in
\eqref{eq:QD-curvature-trace-assumptions}.

Since $v = \omega$, $|v| =1$, and
$c_0 (x, \omega) = - \langle \sigma(x) , \omega \rangle,$
Proposition~\ref{prop:quadratic-gradient-normal-jacobian} with
$u\equiv0$ and $y=\omega$
gives
$$\left|\det D\Phi_t^0(x,\omega)\right|
\le
t^m
\left[
1+t\bigl(\mu_x-\langle\sigma(x),\omega\rangle\bigr)
\right]_+^n
\left(\frac{h_x(t)}{t}\right)^{N-1}.$$
By the relation between $D\Phi_t^0$ and the normal polar Jacobian established in Corollary~\ref{cor:quantitative-Heintze-Karcher},
$$\left|\det D\Phi_t^0(x,\omega)\right|
=
tJ_\Sigma(x,\omega,t),$$
and hence
\begin{equation*}
J_\Sigma(x,\omega,t)
\le
h_x(t)^{N-1}
\left(
\mu_x-\langle\sigma(x),\omega\rangle+\frac1t
\right)_+^n
\end{equation*}

Using the normal-polar area formula as in the proof of Corollary~\ref{cor:quantitative-tube-volume}, and then enlarging the
$t$-integration interval, we obtain
\begin{align}
\operatorname{vol}(T_R(\Sigma))
\le
\int_\Sigma\int_{S_x^\perp\Sigma}\int_0^R
h_x(t)^{N-1}
\left(
\mu_x-\langle\sigma(x),\omega\rangle+\frac1t
\right)_+^n
\,dt\,d\omega\,d\operatorname{vol}_\Sigma .
\label{eq:shifted-tube-bound}
\end{align}

Fix $\varepsilon>0$. By
Lemma~\ref{lem:decay-shifted-model}, there exists
$T_\varepsilon>0$, independent of $x$, such that
\[
h_x(t)\le (L_\Sigma+\varepsilon)h_1(t)
\qquad
\text{for all }x\in\Sigma,\quad t\ge T_\varepsilon.
\]
Uniformly on the compact normal sphere bundle, the integrand in
\eqref{eq:shifted-tube-bound} is $O(t^{m-1})$ as $t\downarrow0$.
Thus its contribution from $0<t<T_\varepsilon$ is finite and independent
of $R$ for $R>T_\varepsilon$. Therefore
\begin{align}
\operatorname{vol}&(T_R(\Sigma))\notag\\
\le{}&
C_\varepsilon
+
(L_\Sigma+\varepsilon)^{N-1}
\int_\Sigma\int_{S_x^\perp\Sigma}\int_{T_\varepsilon}^R
h_1(t)^{N-1}
\left(
\mu_x-\langle\sigma(x),\omega\rangle+\frac1t
\right)_+^n
\,dt\,d\omega\,d\operatorname{vol}_\Sigma ,
\label{eq:shifted-tube-large}
\end{align}
where $C_\varepsilon$ is independent of $R$.

Set $F(R):=\int_0^R h_1(s)^{N-1}\,ds$, then by L'Hôpital's rule,
\begin{equation}\label{eq: F limit}
F(r+c)/F(r)\to1
\end{equation}
for every fixed $c$. So the ball inclusions argument in  \eqref{eq:tube-asymptotic-volume} gives
    \begin{equation*}
\lim_{R\to\infty}
\frac{\operatorname{vol}(T_R(\Sigma))}
{{N|\mathbb B^N|
\int_0^R h_1(t)^{N-1}\,dt}}
=
\theta_\lambda.
\end{equation*}
The factor
$(\mu_x-\langle\sigma(x),\omega\rangle+1/t)_+^n$
converges uniformly on $S^\perp\Sigma$ as $t\to\infty$.
Its averages with weight $h_1(t)^{N-1}$ have the same limit.
Dividing \eqref{eq:shifted-tube-large} by $F(R)$ and letting
$R\to\infty$ therefore gives
\[
N|\mathbb B^N|\theta_\lambda
\le
(L_\Sigma+\varepsilon)^{N-1}
\int_\Sigma\int_{S_x^\perp\Sigma}
\left(
\mu_x-\langle\sigma(x),\omega\rangle
\right)_+^n
\,d\omega\,d\operatorname{vol}_\Sigma.
\]
Letting $\varepsilon\to 0^+$ proves
\eqref{eq:shifted-Willmore-decay}.

Finally,
Lemma \ref{lem:spherical-beta-integral} gives
\[
\int_{S_x^\perp\Sigma}
\left(-\langle\sigma(x),\omega\rangle\right)_+^n\,d\omega
=
C_{n,m}|\sigma(x)|^n,
\]
where $C_{n,m} := \frac{|\mathbb S^{N-1}|}{|\mathbb S^n|}$.
By the definition~\eqref{eq:decay-error-definition},
\begin{align*}
\int_\Sigma
\int_{S_x^\perp\Sigma}
\left(
\mu_x-\langle\sigma(x),\omega\rangle
\right)_+^n
\,d\omega\,d\operatorname{vol}_\Sigma
=
C_{n,m}
\left(
\int_\Sigma|\sigma|^n\,d\operatorname{vol}_\Sigma
+
\mathcal E_{\lambda,\Sigma}
\right).
\end{align*}
Since
\[
\frac{N|\mathbb B^N|}{C_{n,m}}=|\mathbb S^n|,
\]
inequality~\eqref{eq:shifted-Willmore-decay} is equivalent to
\eqref{eq:Willmore-decay-error-form}. Moreover,
$\mathcal E_{\lambda,\Sigma}\ge0$ because $\mu_x\ge0$ and the map
$r\mapsto r_+^n$ is nondecreasing.
\end{proof}

\begin{lemma}\label{lem:annulus-decay-limit}
{Let $\Sigma$ be compact and nonempty, and let
$\mathcal C_{\alpha,r}$ be the region defined in
Lemma~\ref{lem:ABP-basic-Sobolev}.}
For every fixed $0<\alpha<1$,
    \begin{equation*}
\liminf_{r\to\infty}
\frac{\operatorname{vol}(\mathcal C_{\alpha,r})}
{N|\mathbb B^N|\displaystyle\int_{\alpha r}^r h_1(t)^{N-1}\,dt}
\ge
\theta_\lambda.
\end{equation*}
\end{lemma}
\begin{proof}
If $\theta_\lambda=0$, the conclusion is immediate. Assume
$\theta_\lambda>0$.
It follows from \eqref{eq: F limit} and the definition of $\theta_\lambda$ that as $r\to \infty$,
$$
\frac{\operatorname{vol}\left(B_o\left(r-R_0\right)\right)}
{\operatorname{vol}\left(B_o(r)\right)}
=
\frac{\theta_\lambda N|\mathbb B^N|F\left(r-R_0\right)(1+o(1))}
{\theta_\lambda N|\mathbb B^N|F(r)(1+o(1))}
=
1+o(1).
$$
Similarly,
$\operatorname{vol}(B_o(\alpha r+R_0))/\operatorname{vol}(B_o(\alpha r))\to1$ as $r\to \infty$.

For all sufficiently large $r$,
{$B_o(r-R_0)\setminus
\overline{B_o(\alpha r+R_0)}\subset\mathcal C_{\alpha,r}$.} Therefore,
$$
\begin{aligned}
\operatorname{vol}(\mathcal C_{\alpha,r})
&\geq\operatorname{vol}(B_o(r-R_0))-\operatorname{vol}(B_o(\alpha r+R_0))\\
&=\operatorname{vol}(B_o(r))(1+o(1))-\operatorname{vol}(B_o(\alpha r))(1+o(1))\\
&=N|\mathbb B^N|\theta_\lambda\bigl(F(r)-F(\alpha r)\bigr)+o(F(r)).
\end{aligned}
$$
Since $h_1$ is convex and $h_1(0)=0$, we have $F(\alpha r)\leq\alpha^N F(r)$ and hence $F(r)-F(\alpha r)\geq(1-\alpha^N)F(r)$. Dividing the preceding inequality by
$N|\mathbb B^N|\bigl(F(r)-F(\alpha r)\bigr)
=N|\mathbb B^N|\int_{\alpha r}^r h_1(t)^{N-1}\,dt$
and taking the lower limit proves the result.
\end{proof}

For the Sobolev inequality, set
\[
\mu_\Sigma:=b_1+\int_0^{R_0}\lambda(s)\,ds.
\]
Then $\mu_x\le\mu_\Sigma\le2b_1$ for every $x\in\Sigma$.

\begin{theorem}[Quadratic-decay Michael--Simon inequality]
\label{thm:QD-Sobolev}
Let $(M^N,\overline g)$ be a complete noncompact Riemannian manifold, where $N=n+m$ with $n\geq2$ and $m\geq2$, satisfying $(\mathrm{QD})_{n-1}$ with respect to $o\in M$ and $\lambda$.
Let $\Sigma^n\subset M^N$ be a compact embedded submanifold,
possibly with smooth boundary. Then, for every positive
$f\in C^\infty(\Sigma)$,
\begin{align*}
&\int_\Sigma
\sqrt{|\nabla^\Sigma f|^2+f^2|H|^2}
\,d\operatorname{vol}_\Sigma
+
\int_{\partial\Sigma}f\,d\operatorname{vol}_{\partial\Sigma}
+
n\mu_\Sigma\int_\Sigma f\,d\operatorname{vol}_\Sigma
\notag\\
&\qquad\ge
n
\left(
\frac{N|\mathbb B^N|\theta_\lambda}
{m|\mathbb B^m|L_\Sigma^{N-1}}
\right)^{\frac{1}{n}}
\left(
\int_\Sigma f^{\frac n{n-1}}
\,d\operatorname{vol}_\Sigma
\right)^{\frac{n-1}{n}}.
\end{align*}
\end{theorem}

\begin{proof}
If $\theta_\lambda=0$, the result is immediate. Assume
$\theta_\lambda>0$ and first suppose that $\Sigma$ is connected.
The left-hand side is positive: otherwise
$\partial\Sigma=\varnothing$, $H=0$, and $\mu_\Sigma=0$.
Then $\lambda\equiv0$, contradicting inequality \eqref{eq:shifted-Willmore-decay}.
After multiplying $f$ by a positive constant, we may therefore assume that
\begin{align}
\label{eq:QD-Sobolev-normalization}
\int_\Sigma
\sqrt{|\nabla^\Sigma f|^2+f^2|H|^2}\,d\operatorname{vol}_\Sigma
+
\int_{\partial\Sigma}f\,d\operatorname{vol}_{\partial\Sigma}
+
n\mu_\Sigma\int_\Sigma f\,d\operatorname{vol}_\Sigma
=
n\int_\Sigma f^{\frac n{n-1}}\,d\operatorname{vol}_\Sigma .
\end{align}
The compatibility condition gives a solution $u$ of
\[
\begin{cases}
\operatorname{div}_\Sigma(f\nabla^\Sigma u)
=
nf^{\frac n{n-1}}
-
\sqrt{|\nabla^\Sigma f|^2+f^2|H|^2}
-
n\mu_\Sigma f,
&\text{in }\Sigma,\\
\langle\nabla^\Sigma u,\nu\rangle=1,
&\text{on }\partial\Sigma.
\end{cases}
\]
The boundary condition is omitted when
$\partial\Sigma=\varnothing$.

We use the sets $\Omega$, $\mathcal U$, $\mathcal{C}_{\alpha,r}$ and $A_r^u$ introduced in
Section~\ref{sec:isoperimetric}.
Fix $0<\alpha<1$, $r>0$, and
\[
(x,y)\in
A_r^u\cap\mathcal U
\cap
\left\{(x,y):|\nabla^\Sigma u(x)|^2+|y|^2>\alpha^2\right\}.
\]
Set
$v:=\nabla^\Sigma u(x)+y.$
Then
$\alpha<|v|<1.$
The proof of the covering statement in
Lemma~\ref{lem:ABP-basic-Sobolev} is unchanged. Moreover, the same
calculation as in the first part of that lemma gives
\begin{equation*}
c_u(x,y) = \frac{1}{n}\left(\Delta_\Sigma u(x)-\langle H(x),y\rangle\right)
\le
f(x)^{\frac1{n-1}}-\mu_\Sigma.
\end{equation*}
Since $\mu_x\le\mu_\Sigma$ and $|v|<1$, we have
\[
c_u(x,y)+\mu_x|v|
\le
f(x)^{\frac1{n-1}}.
\]
An application of Lemma~\ref{lem:mixed-direction-nonnegative} with the same angle
$\vartheta$ gives along $\gamma(t)=\exp_x(tv)$
\begin{align*}
\mathcal R_n^u(x,y,t)
&\ge -|v|^2\bigl[n\sin^2\vartheta+(n-1)\cos^2\vartheta\bigr]
\lambda(d(o,\gamma(t)))\\
&\ge -n|v|^2\lambda(d(o,\gamma(t))).
\end{align*}
Here $(\mathrm{QD})_n$ and the ordinary Ricci lower bound follow from
Lemma~\ref{lem:QD-monotonicity}. Thus both assumptions in
\eqref{eq:QD-curvature-trace-assumptions} hold.
By applying Proposition~\ref{prop:quadratic-gradient-normal-jacobian}, we obtain

\begin{equation*}
|\det D\Phi_r^u(x,y)|
\le
r^m
\left(1+r f(x)^{\frac1{n-1}}\right)^n
\left(
\frac{h_x(r|v|)}{r|v|}
\right)^{N-1}.
\end{equation*}

Using the covering statement of
Lemma~\ref{lem:ABP-basic-Sobolev} and the area formula, we have
\begin{equation}\label{eq:vol C}
\operatorname{vol}(\mathcal C_{\alpha,r})\leq \int_{{\Omega}} r^{1-n}
\left(1+r f(x)^{\frac1{n-1}}\right)^n
\int_{\{y\,:\,\alpha^2 < |\nabla^\Sigma u(x)|^2 + |y|^2 <1\}}
\left(
\frac{h_x(r|v|)}{|v|}
\right)^{N-1} dy \,d{\mathrm{vol}}_\Sigma (x).
\end{equation}

For $x\in\Omega$, use polar coordinates $\rho=|y|$ and set
$a=\sqrt{|\nabla^\Sigma u(x)|^2+\rho^2}$. Since $m\ge2$, we obtain
\begin{align*}
    &\int_{\{y\,:\,\alpha^2 < |\nabla^\Sigma u(x)|^2 + |y|^2 <1\}}
\left(
\frac{h_x(r|v|)}{|v|}
\right)^{N-1} dy\\
&=\int_{\mathbb S^{m-1}}\int_{\sqrt{(\alpha^2 - |\nabla^\Sigma u|^2)_+}}^{\sqrt{1- |\nabla^\Sigma u|^2}} \left(\frac{h_x \bigl(r \sqrt{|\nabla^\Sigma u(x)|^2 + \rho^2}\bigr)}{\sqrt{|\nabla^\Sigma u(x)|^2 + \rho^2}}\right)^{N-1}
\rho^{m-1} \, d \rho \, d {\mathrm{vol}}_{\mathbb S^{m-1}}\\
& = |\mathbb S^{m-1}| \int_{\max\{\alpha,|\nabla^\Sigma u(x)|\}}^1 \frac{(h_x (r a))^{N-1}}{a^{N-2}} (a^2 - |\nabla^\Sigma u(x)|^2)^{\frac{m-2}{2}}\, da \\
& \leq |\mathbb S^{m-1}| \int_\alpha^1 \frac{(h_x (r a))^{N-1}}{a^n}\, da\\
& \leq \frac{|\mathbb S^{m-1}|}{\alpha^n r} \int_{r\alpha}^{r} (h_x(s))^{N-1}\, ds.
\end{align*}

Fix $\varepsilon>0$ and take $r\ge T_\varepsilon/\alpha$, so that
Lemma~\ref{lem:decay-shifted-model} applies throughout $[\alpha r,r]$.
Combining with \eqref{eq:vol C}, we obtain
\begin{equation}\label{eq:QD-Sobolev-annulus-upper}
    \operatorname{vol}(\mathcal C_{\alpha,r})\leq \frac{|\mathbb S^{m-1}|}{\alpha^n}(L_\Sigma + \varepsilon )^{N-1}
\left( \int_{r\alpha}^{r} (h_1(s))^{N-1}\, ds \right)
\int_\Sigma\left(\frac{1}{r}+ f(x)^{\frac1{n-1}}\right)^n d {\mathrm{vol}}_\Sigma (x)
\end{equation}

Dividing \eqref{eq:QD-Sobolev-annulus-upper} by
$\int_{\alpha r}^r h_1(t)^{N-1}\,dt$ and taking $\liminf_{r\to\infty}$, in view of Lemma~\ref{lem:annulus-decay-limit}, we obtain
\begin{align*}
N|\mathbb B^N|\theta_\lambda
\leq \frac{|\mathbb S^{m-1}|}{\alpha^n}(L_\Sigma + \varepsilon )^{N-1} \int_\Sigma f(x)^\frac{n}{n-1} d {\mathrm{vol}}_\Sigma (x).
\end{align*}

{Letting $\varepsilon\downarrow0$ and then
$\alpha\uparrow1$ gives}
$$\int_\Sigma f(x)^\frac{n}{n-1} d{\mathrm{vol}}_\Sigma (x) \geq \frac{N|\mathbb{B}^N| \theta_\lambda}{|\mathbb S^{m-1}| L_\Sigma^{N-1}}.$$

Since $|\mathbb S^{m-1}|=m|\mathbb B^m|$, this and
\eqref{eq:QD-Sobolev-normalization} give the stated inequality for the
normalized function. From this, Theorem \ref{thm:QD-Sobolev} follows in the case where $\Sigma$ is connected.

If $\Sigma$ is disconnected, apply the connected result to each component
$\Sigma_j$. Since $\mu_{\Sigma_j}\le\mu_\Sigma$ and
$L_{\Sigma_j}\le L_\Sigma$, each component also satisfies the inequality
with the global constants. Summing and using
\[
\sum_j\left(\int_{\Sigma_j}f^{\frac n{n-1}}\right)^{\frac{n-1}{n}}
\ge
\left(\int_\Sigma f^{\frac n{n-1}}\right)^{\frac{n-1}{n}}
\]
proves the result.
\end{proof}

\begin{remark}
Since $\mu_\Sigma\le2b_1$, the inequality remains valid with
$2nb_1\int_\Sigma f$ in place of $n\mu_\Sigma\int_\Sigma f$.
When $\lambda\equiv0$, we have $\mu_\Sigma=0$, $L_\Sigma=1$, and
$\theta_\lambda=\theta$, so Theorem~\ref{thm:QD-Sobolev} reduces to
Theorem~\ref{thm:Sobolev}.
\end{remark}


\begin{thebibliography}{10}

\bibitem{Abresch1985}
U. Abresch,
\newblock Lower curvature bounds, Toponogov's theorem, and bounded topology,
\newblock \emph{Annales scientifiques de l'\'Ecole Normale Sup\'erieure}
\textbf{18} (1985), no.~4, 651--670.

\bibitem{AFM}
V. Agostiniani, M. Fogagnolo, \& L. Mazzieri,
Sharp geometric inequalities for closed hypersurfaces in manifolds with nonnegative Ricci curvature {\em Inventiones Mathematicae}. \textbf{222}, 1033-1101 (2020)

\bibitem{Brendle-ABP}
S. Brendle,
\emph{Geometric inequalities and the Alexandrov--Bakelman--Pucci technique},
Sci. China Math. (2026),

\bibitem{Brendle2023}
S. Brendle,
\newblock Sobolev inequalities in manifolds with nonnegative curvature,
\newblock \emph{Communications on Pure and Applied Mathematics}
\textbf{76} (2023), no. ~9, 2192--2218.

\bibitem{ChongLuoLu2025}
T. Chong, H. Luo, and L. Lu,
\newblock Sobolev inequality in manifolds with lower quadratic curvature decay,
\newblock arXiv preprint arXiv:2503.19254, 2025.

\bibitem{DongLinLu2024}
Y. Dong, H. Lin, and L. Lu,
\newblock Sobolev inequalities in manifolds with asymptotically nonnegative curvature,
\newblock \emph{Calculus of Variations and Partial Differential Equations}
\textbf{63} (2024), no.~4, Paper No.~110, 18 pp.

\bibitem{flaim2024diameter}
M. Flaim and C. Scharrer,
\newblock Diameter estimates for surfaces in conformally flat spaces,
\newblock \emph{Manuscripta Mathematica}. \textbf{174}, 1005-1014 (2024)

\bibitem{HeintzeKarcher1978}
E. Heintze and H. Karcher,
\newblock A general comparison theorem with applications to volume estimates for submanifolds,
\newblock \emph{Annales scientifiques de l'\'Ecole Normale Sup\'erieure}
\textbf{11} (1978), no. ~4, 451--470.

\bibitem{JiKwong}
M. Ji and K.-K. Kwong,
\newblock Fenchel--Willmore inequality for submanifolds in manifolds with non-negative $k$-Ricci curvature,
\newblock arXiv preprint arXiv:2507.07655, 2025.

\bibitem{Jost2017}
J.~Jost,
\emph{Riemannian Geometry and Geometric Analysis},
7th ed., Universitext, Springer, Cham, 2017.

\bibitem{LeeRicci2024LogSobolev}
J. Lee and F. Ricci,
\newblock The Log-Sobolev inequality for a submanifold in manifolds with asymptotic non-negative intermediate Ricci curvature,
\newblock \emph{Journal of Geometric Analysis} \textbf{34} (2024), Paper No.~141.

\bibitem{LottShen2000}
J. Lott and Z. Shen,
\newblock Manifolds with quadratic curvature decay and slow volume growth,
\newblock \emph{Annales scientifiques de l'\'Ecole Normale Sup\'erieure}
\textbf{33} (2000), no.~2, 275--290.

\bibitem{MaWu2024}
H. Ma and J. Wu,
\newblock Sobolev inequalities in manifolds with nonnegative intermediate Ricci curvature,
\newblock \emph{Journal of Geometric Analysis}
\textbf{34} (2024), Article 63.

\bibitem{paeng2014}
S. Paeng,
Diameter of an immersed surface with boundary. {\em Differential Geometry And Its Applications}. \textbf{33} pp. 127-138 (2014)

\bibitem{PanYi}
S. Pan and C. Yi,
\emph{A comparison theorem with applications to sharp geometric inequalities for submanifolds},
arXiv:2605.06074, 2026.

\bibitem{pigola2008vanishing}S. Pigola, A. Setti \& M. Rigoli, Vanishing and finiteness results in geometric analysis: a generalization of the Bochner technique. (Springer,2008)

\bibitem{wang}
X. Wang, Remark on an inequality for closed hypersurfaces in complete manifolds with nonnegative Ricci curvature. {\em Annales De La Faculté Des Sciences De Toulouse: Mathématiques}. \textbf{32}, 173-178 (2023)

\bibitem{wang2013}
Y. Wang and X. Zhang,
An Alexandroff–Bakelman–Pucci estimate on Riemannian manifolds. {\em Advances In Mathematics}. \textbf{232}, 499-512 (2013)

\bibitem{wu2023diameter}
J. Wu,
\newblock Diameter estimates for submanifolds in manifolds with nonnegative curvature,
\newblock \emph{ Differential Geometry And Its Applications}. \textbf{90} pp. 102048 (2023)

\end{thebibliography}
\end{document}